\documentclass[a4paper,reqno]{amsart}

\usepackage{amsfonts}
\usepackage{amssymb}
\usepackage{latexsym}
\usepackage{mathrsfs}
\usepackage{amsthm}
\usepackage{amsmath}
\usepackage{amscd}
\usepackage{enumerate}
\usepackage[latin1]{inputenc}
\usepackage{a4wide}
\usepackage[usenames]{color}

\usepackage[arc,all]{xy}

\usepackage{color}
\xyoption{2cell}
\UseAllTwocells

\newtheorem{theorem}{Theorem}[section]
\newtheorem{proposition}[theorem]{Proposition}
\newtheorem{corollary}[theorem]{Corollary}
\newtheorem{lemma}[theorem]{Lemma}
\newtheorem*{theorem*}{Theorem}
\newtheorem*{proposition*}{Proposition}
\newtheorem*{corollary*}{Corollary}
\newtheorem*{lemma*}{Lemma}
\theoremstyle{definition}
\newtheorem{definition}[theorem]{Definition}
\newtheorem{punto}[theorem]{}

\newtheorem*{punto*}{ }
\newtheorem{example}[theorem]{Example}
\newtheorem{remark}[theorem]{Remark}
\newtheorem*{remark*}{Remark}

\newtheorem*{definition*}{Definition}

\newcommand{\al}{{\mathfrak{Alg}}}
\newcommand{\au}{{\textbf{Aut}}}
\newcommand{\V}{{\mathcal{V}}}
\newcommand{\VV}{{\mathfrak {Alg}}}

\newcommand{\W}{{\mathfrak {Coalg}}}
\newcommand{\pa}{{\partial}}
\newcommand{\ul}{\underline}
\newcommand{\ol}{\overline}
\newcommand{\nr}{\rightsquigarrow}
\newcommand{\w}{{\circ}}
\newcommand{\ww}{{\otimes}}

\newcommand{\la}{{\lambda}}
\newcommand{\xr}{\xrightarrow}
\newcommand{\A}{{\mathcal{A}}}
\newcommand{\B}{{\mathbb{B}}}

\newcommand{\bH}{{\mathcal{H}}}
\newcommand{\X}{{\mathcal{X}}}

\newcommand{\e}{{\varepsilon}}

\begin{document}
\title{Some exact sequences associated with adjunctions in bicategories. Applications.}
\author{J. G{\'o}mez-Torrecillas \\ and \\ B. Mesablishvili}
\address{Departament of Algebra and CITIC, Universidad de
Granada  E18071 Granada, Spain}
\email{gomezj@ugr.es}
\address{
Razmadze Mathematical Institute, Tbilisi 0193, Republic of
Georgia}
\email{bachi@rmi.ge}

\keywords{Bicategories, Adjunctions, (Firm) bimodules, Coalgebras, Amitsur cohomology, Descent data, Hilbert's theorem 90}
\subjclass[2010]{16D20, 16T15, 18D05, 18D10, 18G30}

\thanks{Research partially supported by grant MTM2013-41992-P from the Ministerio de Econom\'{\i}a y Competitividad of the
Spanish Government and from FEDER. The second named author was supported by Shota Rustaveli National
Science Foundation Grants DI/18/5- 113/13 and FR/189/5- 113/14}

\begin{abstract}
We prove that the classical result asserting that the relative Picard group of a faithfully flat extension of commutative rings is isomorphic to the first Amitsur cohomology group stills valid in the realm of symmetric monoidal categories. To this end, we built some group exact sequences from an adjunction in a bicategory, which are of independent interest. As a particular byproduct of the evolving theory, we prove a version of Hilbert's theorem 90 for cocommutatvie coalgebra coextensions (=surjective homomorphisms).
\end{abstract}

\maketitle

\tableofcontents

\section*{Introduction}
We prove that the classical result asserting that the relative Picard group of a faithfully flat extension of commutative rings is isomorphic to the first Amitsur cohomology group stills valid in the realm of symmetric monoidal categories. To this end, we prove that for any commutative algebra $\textbf{A} = (A,m,e)$ in a symmetric monoidal category $\V$ satisfying some technical conditions, there is an exact sequence of groups
\begin{equation}\label{seq.5a}
0 \xr{} \textbf{Aut}_{\V}(I) \xr{\varpi_0} \textbf{Aut}_{\V_\textbf{A}}(A)\xr{\kappa_A}
\textbf{Aut}_{\textbf{A}-\texttt{cor}}(A \ww A) \xr{o_A} \textbf{Pic}^{c}(\textbf{I}) \xr{\textbf{Pic}^{c}(e)} \textbf{Pic}^{c}(\textbf{A}).
\end{equation}
Details on the group homomorphisms involved are to be found in subsection \ref{Picard}.
The sequence \eqref{seq.5a} will be built with the help of an exact sequence of groups associated, under mild conditions, to any adjunction in a bicategory (see Theorem \ref{exact}). The latter generalizes some useful exact sequences associated to a ring extension, used in \cite{Kanzaki:1968, Miyashita:1973} in the unital case, and in \cite{ElKaoutit/Gomez:2012} in the realm of ring with local units, to derive adequate versions of Chase-Harrison-Rosenberg's seven exact sequence \cite{Chase/alt:1965}. Thus, our results could be of independent interest for extending to wider contexts the aforementioned seven terms sequence. We specialize our general theory to the case of the bicategory of bimodules over a (non necessarily symmetric) monoidal category $\V$ (see Theorem \ref{main.alg}). We get in particular an exact sequence of groups associated to a monomorphic homomorphism of $\V$--algebras (see Example \ref{AB}).

Every comonadic homomorphism of commutative $\V$--algebras $\iota: \mathbf{A} \to \mathbf{B}$, where $\V$ is a symmetric monoidal category fulfilling some minimal technical requirements, leads to a homomorphism of abelian groups $\textbf{Pic}^{c}(\iota):\textbf{Pic}^{c}(\textbf{A})\to \textbf{Pic}^{c}(\textbf{B})$. Our theory applies to obtain (Theorem \ref{Amitsur.alg.gen.}) that, if the change-of-base functor associated to $\iota$ is comonadic, then  $\mathbf{Ker}(\mathbf{Pic}^{c}(\iota))$ is isomorphic to the first Amitsur cohomology group $\mathcal{H}^1(\iota\,, \ul{\mathbf{Aut}}_A^{\mathfrak {Alg}})$ of $\iota$ with coefficients in the functor $\ul{\mathbf{Aut}}_A^{\mathfrak {Alg}}$ (which is a generalization of the usual units functor, see Lemma \ref{autalg}). The problem is easily reduced to prove that, in \eqref{seq.5a}, the cokernel of $\kappa_A$ is isomorphic to $\mathcal{H}^1(e\,, \ul{\mathbf{Aut}}_I^{\mathfrak {Alg}})$, whenever $- \otimes A$ is a comonadic functor. Our proof involves some classical results, namely a version of  the B\'{e}nabou-Roubaud-Beck theorem identifying the category of descent data with an Eilenberg-Moore category (Theorem \ref{BRB}), and Grothendiek's  isomorphism between the Amitsur first  cohomology pointed set and the set of descent data of an effective descent morphism (Proposition \ref{Grothendieck}). A brief account of the required classical theory is given in the Appendix.

In the final section, we apply our general theory to the bicategory of bicomodules. As a particular  a version of Hilbert's theorem 90 for cocommutatvie coalgebra coextensions (=surjective homomorphisms) (Theorem \ref{Hilbert}) is obtained.

\section{Preliminaries}

In this section, we list some categorical notions and basic constructions that will be needed. Our basic references on categories are \cite{A,FB,M}.

\begin{punto}{\bf Subobjects and quotient objects.}
Let $a$ be an object of a category $\mathcal{A}$. Preorder monomorphisms with range $a$ by setting $j\leq i$ if
$j$ is of the form $j=ik$; the equivalence classes for the relation
$$``j\leq i \,\,\,\text{and}\,\,\, i\leq j"$$ are called \emph{subobjects}
of $a$. We write $\textbf{Sub}_{\mathcal{A}}(a)$ for the the class of all
subobjects of $a$. We often identify a subobject with a representative monomorphism, and we call
the subobject regular etc. if the monomorphism $i$ is regular etc.

Dually, one has the collection $\textbf{Quot}_{\mathcal{A}}(a)=\textbf{Sub}_{\mathcal{A}^{op}}(a)$ of isomorphism classes
of epimorphisms with domain $a$ ($\mathcal{A}^{op}$ denotes the opposite category of $\mathcal{A}$). We shall call an element of $\textbf{Quot}_{\mathcal{A}}(a)$ a \emph{quotient object}
of $a$. Note that for epimorphisms with domain $a$ we write $j\leq i$ if $j$ is of
the form $j=ki$.
\end{punto}

\begin{punto}{\bf Images and coimages.} Recall that a category admits \emph{images} if any morphism $f$ can
be written as $f=ip$ with $i$ monomorphic and $p$ regular epimorphic. The subobject $[i]$ of the codomain of
$f$ is called the \emph{image} of $f$. Dually, a category is said to admit \emph{coimages} if any morphism $f$
can be written as $f=ip$ with $p$ epimorphic and $i$ regular monomorphic. The quotient object $[p]$ of the domain
of $f$ is called the \emph{coimage} of $f$. We say that a monoidal category admits (co)images if its underlying
ordinary category does so.
\end{punto}

\begin{punto}\label{medioinvertibles}{\bf Subobjects and quotient objects of (co)algebras.}
Suppose that $\V=(\V,\otimes, I)$ is a fixed monoidal category with underlying ordinary category $\V$,
tensor product $\otimes$ and monoidal unit $I$. Recall that an \emph{algebra} in  $\V$ (or $\V$-algebra)
consists of an object $A$ of $\V$ endowed with a multiplication $m_A: A \otimes A \to A$  and unit morphism $e_A : I \to A$,
subject to the usual associative and identity conditions. These algebras are the objects of a category  $\verb"Alg"(\V)$ with the obvious morphisms.
Dually, one has the notions of $\V$--\emph{coalgebra}; the corresponding category of $\V$--coalgebras is denoted by $\verb"Coalg"(\V)$.

Given a $\V$-algebra $\textbf{A}=(A,m_A,e_A)$, we write $\mathfrak{I}^l_\V(\mathbf{A})$ for the subset
of $\textbf{Sub}_\V(A)$ consisting of those elements $[(J,i\!_J : J \to A)]$ for which the composite
$$\xi^l_{i\!_{_J}}:A\ww J \xr{A\ww i\!_J}A\ww A \xr{m_A}A$$ is an isomorphism. Symmetrically, we let
$\mathfrak{I}^r_\V(\textbf{A})$ denote the subclass of $\textbf{Sub}_\V(A)$ consisting of those elements
$[(J,i\!_J : J \to A)]$, for which the composite $$\xi^r_{i\!_{_J}}:J\ww A \xr{i\!_J \ww A}A\ww A \xr{m_A}A$$
is an isomorphism.

Dually, for a $\V$-coalgebra $\textbf{C}=(C,\delta,\e)$, we write $\mathcal{Q}_\V^l(\textbf{C})$ (resp.
$\mathcal{Q}_\V^r(\textbf{C})$) for the subset of $\textbf{Quot}_\V(C)$ consisting of those
elements ${[(P,\pi_P: C\to P)]}$ for which the composite $$C \xr{\delta}C \ww C \xr{C \ww \pi_P} C \ww P$$
(resp. $$C \xr{\delta}C \ww C \xr{\pi_P \ww C} P \ww C)$$ is an isomorphism.

\end{punto}

\begin{punto}{\bf Adjunctions in bicategories.}
We begin by recalling
from \cite{Bn}  that a bicategory $\mathbb{B}$ consists of :
\begin{itemize}
\item a class $\text{Ob}(\mathbb{B})$ of objects, or 0-cells;

\item a family $\mathbb{B}(A, B)$, for all $A, B \in
\text{Ob}(\mathbb{B})$, of hom-categories, whose objects and
morphisms are respectively called 1-cells and 2-cells;

\item a (horizontal) composition operation, given by a family of
functors
$$\mathbb{B}(B, C) \times \mathbb{B}(A, B) \to \mathbb{B}(A,
C)$$ whose action on a pair $(g, f) \in \mathbb{B}(B, C)\times
\mathbb{B}(A, B)$ is written $g \w f$;

\item identities, given by 1-cells $1_A \in \mathbb{B}(A, A)$,
for $A \in \text{Ob}(\mathbb{B})$;

\item natural isomorphisms $$\alpha_{h, g, f}: (h \w g)
\w f \simeq h \w(g \w f), l_f : 1_A \w f \simeq f \,\,\text{and}
\,\, r_f : f \w 1_A \simeq f, $$
\end{itemize} subject to two coherence axioms (see \cite{Bn}).

When the context is clear, we write
$[A,B]$ instead of $\B(A,B)$.

We review the concept of adjunction in an arbitrary bicategory along with some of the general theory
needed later on.

Fix a bicategory $\B$. An adjunction $(\eta, \e: f \dashv f^*: B \to A)$ in $\B$ consists of objects $A$ and $B$,
1-cells $f: A\to B$ and $f^*: B\to A$, and 2-cells $\eta:1_A\to f^*\w f$, called the
\emph{unit}, and $\e: f\w  f^* \to 1_B$, called the \emph{counit} such
the following diagrams commute in $[A, B]$ and $[B, A]$, respectively:
\begin{equation}\label{adj1} \xymatrix{ f \w 1_{A} \ar[r]^{f \w \eta}\ar[d]_{r_f}
& f \w (f^* \w f)  \ar[r]^{\alpha_{f, f^*, f}^{-1}}&
(f \w f^*)\w f \ar[d]^{\e \w f}\\
f \ar[rr]_{l_f^{-1}}&& 1_B \w f}
\end{equation}
and
\begin{equation}\label{adj2}\xymatrix{ 1_A \w f^*  \ar[r]^{\eta \w f^* }
\ar[d]_{l_{f^*}}& (f^* \w f )\w f^* \ar[r]^{\alpha_{f^*, f, f^*}}
& f^* \w (f \w f^*) \ar[d]^{ f^* \w \e } \\
f^* \ar[rr]_{r_{f^*}^{-1}}&& f^* \w 1_B \,.}
\end{equation}

Let $\eta, \e :f \dashv f^* : B \to A$ be adjunction in
$\B$ and let $X$ be an arbitrary 0-cell of $\B$.
Then the functor
$$[X,f]=f \w - :[X, A] \to [X, B]$$ admits as a right
adjoint the functor $$[X,f^*]=f^* \w - : [X, B] \to [X, A]\,.
$$ The unit $\eta^X$ and counit $\e^X$ of this adjunction are
given by the formulas:
$$\eta^X_g : g \xr{l_g^{-1}} 1_A \w g \xr{\eta \w g} (f^* \w f) \w g \xr{\alpha_{f^*, f, g}}
f^* \w (f \w g),\,\,\, \text{for\,\,all}\,\, g \in [X,A]$$ and
$$\e^X_h : f \w (f^* \w h)\xr{\alpha_{f,f^*, h}^{-1}}(f\w f^*) \w h \xr{\e \w h} 1_B \w h \xr{l_h} h, \,\,\,
\text{for\,\,all}\,\, h \in [X, B].$$

The situation may be pictured as
\[
\xymatrix{ [X,A] \ar@{}[rr]|\bot \ar@/^1pc/ [rr]^{[X,f]=f\w -} && [X,B]
\ar@/^1pc/ [ll]^{[X, f^*]=f^*\w -} }
\]

\end{punto}

\begin{definition*} A 1-cell $f:A \to B$ in $\mathbb{B}$ is called \emph{invertible} if there
exist a 1-cell $g : B \to A$ and isomorphisms $g\w f \simeq 1_A$ and $f\w g \simeq 1_B$. The 1-cell $g$ is called \emph{a pseudo-inverse} of $f$.
\end{definition*}

Recall that an \emph{adjoint equivalence} in $\mathbb{B}$  is an adjunction in which
both the unit and counit are isomorphisms, and that any equivalence is part of an adjoint equivalence.

\begin{remark}\label{equivalence} If a 1-cell $h: A \to A$ is invertible, then, for any object $X \in \mathbb{B}$, both
functors $[X,h]=h \w - :[X, A] \to [X, A]$ and $[h,X]=- \w h :[A, X] \to [A, X]$ are equivalences of categories, and thus they preserve existing limits and colimits. In particular, they preserve monomorphisms and epimorphisms.
\end{remark}

The following is an example of bicategory to which some of our general results will be applied.

\begin{example}\label{Firm}{\bf Firm bimodules.}
Let $S$ be a ring, which is not assumed to be unital. A right $S$--module $M$ is said to be \emph{firm} \cite{Quillen} if the map $M \otimes_S S \to M$ sending $m \otimes_S s$ to $ms$ is an isomorphism. Thus, the ring $S$ is said to be firm  if the multiplication map $S \otimes_S  S \to S$ is an isomorphism. Firm left modules and firm bimodules are defined analogously. We denote by $\mathbf{Firm}$ the bicategory whose $0$--cells are firm rings, the $1$--cells are firm bimodules and the $2$--cells are homomorphisms of firm bimodules. The horizontal composition in $\mathbf{Firm}$ is given by the tensor product of bimodules. Given a homomorphism $\varphi : R \to S$, where $R$ and $S$ are firm rings, we may consider the bimodules ${}_R S_S$ and ${}_S S _R$ in the usual way. We say that $\varphi$ is a \emph{homomorphism of firm rings} if ${}_RS_S$ and ${}_SS_R$ are firm bimodules. In this case, we have $1$--cells ${}_S S _R : R \to S$ and ${}_R S _S : S \to R$, which form an adjunction ${}_SS_R \dashv {}_RS_S$ in $\mathbf{Firm}$. Its counit is the multiplication map $\mu : S \otimes_R S \to S$, while the unit is given by the composite $R \xr{\varphi} S \xr{\nu} S \otimes_S S$, where $\nu$ denotes the inverse of the multiplication map $S \otimes_S S \xr{\cong} S$.
\end{example}

\begin{punto}{\bf Mates.}
Recall from \cite{KS} that for adjunctions $(\eta, \e: f \dashv g: B \to A)$ and $(\eta', \e': f' \dashv g': B \to A)$ in $\B$,
there is a bijection between 2-cells $$\sigma : f \to f' \,\,\, \text{and}  \,\,\, \ol{\sigma}:g' \to g,$$
where $\ol{\sigma}$ is obtained as the composite
$$g' \stackrel{l^{-1}_{\!g'}}\simeq 1_A \w g'\xr{\eta \w g'} (g\w f)\w g' \xr{(g\w \sigma)\w g'}(g\w f')\w g'\stackrel{\alpha}\simeq
g\w (f'\w g')\xr{g\w \e'}g\w 1_B \stackrel{r_{g}}\simeq g$$
and $\sigma$ is given as the composite
$$f \stackrel{r_{f}}\simeq f \w 1_A \xr{f \w \eta'} f \w (g'\w f') \xr{ f \w (\ol{\sigma}\w f')} f \w (g\w f')\stackrel{\alpha^{-1}}\simeq
(f\w g)\w f')\xr{\e \w f'}  1_B \w f' \stackrel{l^{-1}_{\!f'}}\simeq f'.$$

In this situation, $\sigma$ and $\ol{\sigma}$ are called \emph{mates} under the given adjunctions and this is
denoted by $\sigma \dashv \ol{\sigma}$.

\begin{lemma} \label{mates} If $\sigma \dashv \ol{\sigma}$ under adjunctions $(f \dashv g: B \to A)$ and $(f' \dashv g': B \to A)$. Then
 $\sigma$ is an isomorphism iff  \,\,$\ol{\sigma}$ \,is.
\end{lemma}
\end{punto}

\section{Invertible cells associated to an adjunction}
Let $A$ be an object of a bicategory $\mathbb{B}$. We call a (co)algebra in
the monoidal category $[A,A]$ an $A$-\emph{\emph{(}co\emph{)}ring} and write $A\texttt{-ring} = \verb"Alg"([A,A]) $ (resp. $A\texttt{-cor} = \verb"Coalg"([A,A])$) for the category
of $A$-(co)rings.

Any 1-cell with a right adjoint generates a ring as well as a coring as follows. If $\eta_f, \e_f :
f \dashv f^* : B \to A$ is an adjunction in $\mathbb{B}$, then the triple
\begin{equation}\label{genring}
\mathcal{S}_f=(f^* \w f, m_f ,\eta_f ),
\end{equation}
where $m_f$ is the composite $$(r_{f^*} \w f  )\cdot((f^* \w \e_f) \w f)\cdot( \alpha_{f^*, f,
f^*}\w f)\cdot(\alpha_{f^* \w f,f^*, f })^{-1}:(f^* \w f )\w (f^* \w f)\to f^* \w f, $$ is an $A$-ring, while the triple
$$\mathfrak{C}_f =(f \w f^*, \delta_f , \e_f), $$ where $\delta_f $ is the composite
$$(\alpha_{f^* \w f,f^*, f })\cdot (\alpha_{f, f^*, f}^{-1}\w f^*)
\cdot((f \w \eta_f) \w f^* )\cdot(r_f^{-1}\w f^*):f \w f^* \to (f\w f^*)\w (f \w f^*),$$ is a $B$-coring.

Since $\mathcal{S}_f$ is an algebra in the monoidal category $[A,A]$, one has the sets $\mathfrak{I}^l_{[A,A]}(\mathcal{S}_f)$
and $\mathfrak{I}^r_{[A,A]}(\mathcal{S}_f)$.
Recall from \cite[Remark 4.2]{GM} that for any monomorphic 2-cell $i_h:h\to f^*\w f$,
$$\xi^l_{i_h}=(f^* \w \xi_{i_h})\cdot\alpha_{f^*,\,f,\,h}$$
and
$$\xi^r_{i_h}=(\xi^*_{i_h} \w f)\cdot\alpha^{-1}_{h,\,f^*,\,f},$$
where $\xi_{i_h}$ and $\xi^*_{i_h}$ are the composites
$$\xi_{i_h} :f \w h \xr{f \w i_h} f\w (f^* \w f) \xr{\alpha^{-1}_{f, f^*,f}}(f \w f^*) \w f \xr{\e_f \w f}1_B \w f \xr{l_f}f$$
and
$$\xi^*_{i_h} :h \w f^* \xr{ i_h \w f^*} (f^*\w f) \w f^* \xr{\alpha_{f^*, f,f^*}}f^*\w (f \w f^*) \xr{ f^*\w \e_f}f^* \w 1_B \xr{r_{f^*}}f^*,$$
respectively.

\bigskip

We write  $\mathfrak{I}^{A,\,l}_f$ (resp. $\mathfrak{I}^{A,r}_f$) for the subset of $\mathfrak{I}^l_{[A,A]}(\mathcal{S}_f)$ (resp. $\mathfrak{I}^r_{[A,A]}(\mathcal{S}_f)$) determined by those subobjects $[(h, i_h)]$ with $h$ invertible.

\begin{proposition}\label{lifting} Let $\eta_f, \e_f :f \dashv f^* : B \to A$ be an adjunction in $\mathbb{B}$
such that $\eta_f$ is monomorphic in $[A,A]$ and $h: A \to A$ be an invertible 1-cell. If
there is an isomorphism $\sigma : f\w h\to f$ in $[A,B]$, then $[(h,i_h)] \in \mathfrak{I}^{A,\,l}_f,$
where $i_h$ is the composite $h \xr{\eta_f \w h}f^* \w f \w \,h \xr{f^* \w \sigma} f^* \w f$ \footnote{
For simplicity of exposition we sometimes treat $\B$ as a 2-category which is justified by the coherence theorem (see \cite{McP}) asserting
that every bicategory is biequivalent to a 2-category. Consequently, we sometimes omit brackets in the horisontal compositions and suppress the associativity constraints $\alpha$ and the unitality constraints $l$ and $r$.}.
\end{proposition}
\begin{proof}Suppose that $h$ is invertible and that there is an isomorphism $\sigma : f\w h \to f$ in $[A,B]$.
Since $\eta_f$ is assumed to be monomorphic in $[A,A]$, it follows from Remark \ref{equivalence} that
$h \xr{\eta\!_f \w h}f^* \w f \w \,h$ is monomorphic in $[A,A]$. Then, since $\sigma$ is an isomorphism, $i_h$
must be a monomorphism too. Now, since the functor $f^* \w-:[A,B]\to [A,A]$ is right adjoint for the functor
$ f\w-:[A,A]\to [A,B]$ with $\eta_f \w -$ as unit and $\e_f \w -$ as counit, it follows that
$\sigma=(\e_f \w f) \cdot(f\w i_h)=\xi_{i_h}$. Therefore, $\xi^l_{i_h}=f^* \w \xi_{i_h}=f^* \w \,\sigma$ is an
isomorphism too and hence $[(h, i_h)]\in \mathfrak{I}^l_{[A,\,A]}(\mathcal{S}_f)$.
\end{proof}

\begin{proposition}\label{lifting.1} In the situation of Proposition \ref{lifting}, suppose that $h^*$ is a pseudo-inverse
of $h$. Then there is a monomorphic 2-cell $i_{h^*}:h^* \to f^* \w f$ such that $[(h^*, i_{h^*})]\in \mathfrak{I}^{A,r}_f$.
\end{proposition}
\begin{proof} Composing the adjunction $f \dashv f^*$ with $h \dashv h^*$  yields an adjunction $f\w h \dashv h^* \w f^*$.
Since $f \w h \simeq f$ and since adjoints are unique up to unique isomorphism, one has an isomorphism $\tau:h^* \w f^* \simeq f^*.$
Now, if we take $i_{h^*}$ to be the composite $i_{h^*}:h^* \xr{h^* \w \,\eta_f }h^* \w f^* \w f \xr{\tau \w f} f^* \w f$, then
the result is proved in exactly the same way as Proposition \ref{lifting}, but this time using the adjunction
$$-\w f^* \dashv -\w f:[B,A]\to [A,A].$$
\end{proof}

\begin{proposition}\label{conjugate} Let $\eta, \e: f \dashv f^* : B \to A$ be an adjunction and
$(\eta_h, \e_h : h \dashv h^* :A \to A)$ be an adjoint equivalence in $\B$.
Then for any 2-cell $i_h: h \to f^* \w f$,  the following are equivalent:
\begin{itemize}
  \item [(i)]$\xi_{i_h}:f \w h  \to f$ is an isomorphism;
  \item [(ii)] $\xi^l_{i_h}: (f^* \w f)\w h \to f^* \w f$ is an isomorphism;
  \item [(iii)] $\xi^*_{i_h}:h\w f^* \to f^*$ is an isomorphism;
  \item [(iv)] $\xi^r_{i_h}: h \w (f^* \w f) \to f^* \w f$ is an isomorphism.
\end{itemize}
Moreover, $i_h$ is monomorphism in $[A,A]$ provided any (and hence all) of the above conditions hold.
\end{proposition}
\begin{proof} Since (i) is equivalent to (ii) and (iii) is equivalent to (iv) by \cite[Remark 4.2]{GM} and its dual,
we have only to show that (i) and (iii) are equivalent.

Note first that composing the adjunction  $(\eta_f, \e_f:f \dashv f^*)$ with $(\eta_h, \e_h: h \dashv h^*)$
yields an adjunction $(\ol{\eta}, \ol{\e}:f \w \,h \dashv h^* \w f^*)$, where $\ol{\eta}$ and $\ol{\e}$  are the
composites $$1_A \xr{\eta_h}h^* \w \,h \xr{h^* \w \eta_f \w h}h^* \w f^* \w f \w  \,h$$ and
$$f \w \, h \w \,h^* \w f^* \xr{f \w \e_h \w f^*}f \w f^*  \xr{\e_f } 1_B,$$ respectively.

Consider now the composite
$$\overline{\xi}_{i_h}:f^* \xr{\ol{\eta} \w f^*} h^* \w f^* \w f \w  \,h\, \w f^* \xr{h^* \w f^* \w \xi_{i_h} \w f^*}
h^* \w f^* \w f \w f^* \xr{ h^* \w f^* \w \e_f } h^* \w f^*,$$ which is the mate of $\xi_{i_h}$ under the adjunctions
$(f \w \,h \dashv h^* \w f^*)$ and $(f \dashv f^*)$. A straightforward calculation, using the expression for $\ol{\eta}$
and $\xi_{i_h}$, shows that $\overline{\xi}_{i_h}$ is the composite
$$f^* \xr{\eta_h \w f^*}h^* \w h \w f^* \xr{h^* \w i_h \w f^*} h^* \w f^* \w f \w f^* \xr{ h^* \w f^* \w \e_f}  h^* \w f^*,$$
and therefore
$$
\overline{\xi}_{i_h}=(h^* \w \xi^*_{i_h})\cdot (\eta_h \w f^*),
$$ implying -- since both $h^*$ and $\eta_h$ are invertible 1-cells
-- that $\overline{\xi}_{i_h}$ is an isomorphism iff $\xi^*_{i_h}$ is. In the light of
Lemma \ref{mates} one now concludes that (i) and (iii) are equivalent.

Finally, each of the conditions (i)-(iv) implies that $\xi^l_{i_h}$ is an isomorphism, and then $i_h$ is a monomorphism
by Proposition \ref{lifting}. This completes the proof.

\end{proof}

 \begin{proposition} \label{left=right} $\mathfrak{I}^{A,\,l}_f=\mathfrak{I}^{A,\,r}_f$.
\end{proposition}
\begin{proof} By symmetry, it suffices to prove the inclusion $\mathfrak{I}^{A,\,l}_f \subseteq\mathfrak{I}^{A,\,r}_f$. To this end
consider an arbitrary element $[(h, i_h)] \in\mathfrak{I}^{A,\,l}_f$. Since $h $ is an invertible 1-cell, we need only show that
$[(h, i_h)] \in \mathfrak{I}^r_{[A,A]}(\mathcal{S}_f)$. Since $[(h, i_h)] \in \mathfrak{I}^{A,\,l}_f \subseteq\mathfrak{I}^l_{[A,A]}(\mathcal{S}_f)$,          the 2-cell $\xi^l_{i_h}:  (f^* \w f) \w h  \to f^* \w f$ is an isomorphism, and then $\xi_{i_h}:f \w h \to h$ is also an
isomorphism by Remark \cite[Remark 4.2]{GM}. Applying now Proposition \ref{conjugate}
gives that both $\xi^*_{i_h}:h \w f^* \to f^*$ and $\xi^r_{i_h}=\xi^*_{i_h}\w f:h \w (f^*\w f) \to f^*\w f$
are isomorphisms. Thus, $[(h, i_h)] \in \mathfrak{I}^r_{[\A,\A]}(\mathcal{S}_f)$, and hence $[(h, i_h)] \in\mathfrak{I}^{A,\,r}_f$.
\end{proof}

\begin{definition} We write $\mathfrak{I}^A_f$ to denote either $\mathfrak{I}^{A,\,l}_f$ or $\mathfrak{I}^{A,\,r}_f$.
\end{definition}

\section{Exact sequences of groups related to adjunctions in bicategories}\label{SucesionesExactas}
Fix an adjunction $\eta_f, \e_f :f \dashv f^* : B \to A$ in $\mathbb{B}$. In this section we suppose,
with the exception of Subsection \ref{duality}, that $\mathbb{B}$ is a bicategory such that each hom-category admits
finite limits and images\footnote{Indeed, we need this assumption only for the hom-category $[A,A]$.} and that the
2-cell $\eta_f: 1_A \to f^* \w f$ is a monomorphism in $[A,A]$. In this case, $\textbf{Sub}_{[A,A]}(f^* \w f)$ has a
monoid structure formed as in \cite[Proposition 3.2]{GM}, and  throughout this paper, when considering
$\textbf{Sub}_{[A,A]}(f^* \w f)$ as a monoid, we always mean this monoid structure.

\subsection{Automorphisms and invertible subobjects}\label{autoinvert}

One can easily verify that the assignment taking a 2-cell $s:1_A \to 1_A$ to the composite
$$f^* \xr{l_{f^*}^{-1}} 1_A \w f^* \xr {s \w f^*} 1_A \w f^* \xr{l_{f^*}}  f^*,$$
yields a monoid morphism
$$\varpi:[A,A](1_A, 1_A) \to [B,A](f^*,f^*),$$
which gives, by restriction, a homomorphism of groups
 $$\varpi_0:\text{\textbf{Aut}}_{[A,\,A]}(1_A) \to \text{\textbf{Aut}}_{[B,\,A]}(f^*).$$
 between the groups of automorphisms of the objects $1_A$ and $f^*$, respectively.

\begin{proposition}\label{mono} 
The map $\varpi_0$ is a monomorphism of  groups.
\end{proposition}
\begin{proof} If $s \in \text{\textbf{Aut}}_{[A,\,A]}(1_A)$ is such that $\varpi_0(s)=1_{f^*}$,
then $1_{f^*}=l_{f^*}\cdot (s \w f^*)\cdot l_{f^*}^{-1}$ and hence $l_{f^*}= l_{f^*}\cdot (s \w f^*)$.
But since $l_{f^*}= l_{f^*}\cdot (1\!_{1\!_A} \w f^*)$ and since $l_{f^*}$ is invertible, it follows that
$1\!_{1\!_A} \w f^*=s \w f^*$ and hence $1\!_{1\!_A} \w (f^*\w f)=s \,\w (f^*\w f)$. Direct calculation then shows that
$\eta_f \cdot 1\!_{1\!_A}=\eta_f \cdot s$. Now, since $\eta_f$ is
assumed to be monomorphic, the map $$[A,A](1_A, \eta_f):[A,A](1_A,1_A)\to [A,\,A][1_A, f^*\w f]$$ is injective, implying
that $1\!_{1\!_A}=s$. Thus, $\varpi_0$ is a monomorphism of  groups.
\end{proof}

For any $\lambda \in \text{\textbf{Aut}}_{[B,\,A]}(f^*)$,
form the pullback

\begin{equation}\label{ex}
\xymatrix{ f_\lambda \ar[d]_{p_\lambda} \ar[r]^-{i_\lambda}& f^* \w
f \ar[d]^{\lambda \w f}\\
1_A \ar[r]_-{\eta_f}& f^* \w f\,.}
\end{equation}

Since, by hypothesis, $\eta_f$ is a monomorphism in $[A,\,A]$, so too is $i_\lambda$, and thus
$(f_\lambda, i_\lambda)$ represents an element of $\text{Sub}_{[A,\,A]}(f\w f^*)$, implying
$-$ since pullbacks are unique up to isomorphism $-$ that the assignment
$\lambda \longmapsto [(f_\lambda, i_\lambda)]$  yields a map
$\mathcal{D}_f:\text{\textbf{Aut}}_{[B,\,A]}(f^*) \to
\text{\textbf{Sub}}_{[A,\,A]}(f^*\w f).$

Consider now the diagram
$$
\xymatrix @C=0.3in @R=0.7in{(f^* \w f) \w(f^*\w f) \ar@{}[rd]|{(1)}
\ar[r]^-{\alpha^{-1}}\ar[d]|{(\lambda \w f)\w (f^* \w f)}& ((f^* \w
f)\w f^*)\w f  \ar@{}[rd]|{(2)} \ar[r]^-{\alpha \w f}\ar[d]|{((\lambda \w f) \w f^*)\w f}&
(f^* \w (f \w f^*))\w f \ar@{}[rrd]|{(3)} \ar[d]|{(\lambda \w (f \w f^*))\w
f}\ar[rr]^-{(f^* \w \varepsilon_f)\w f} && (f^* \w 1_A)\w f \ar@{}[rd]|{(4)}
\ar[r]^-{r_{f^*} \w
f}\ar[d]|{(\lambda \w 1_A) \w f}&f^* \w f \ar[d]|{\lambda \w f}\\
(f^* \w f) \w(f^*\w f) \ar[r]_-{\alpha^{-1}}&((f^* \w f)\w f^*)\w f
\ar[r]_-{\alpha \w f}&(f^* \w (f \w f^*))\w f \ar[rr]_-{(f^* \w
\varepsilon_f)\w f} && (f^* \w 1_A)\w f \ar[r]_-{r_{f^*} \w f}&f^* \w
f}$$ in which rectangles (1) and (2) commute by naturality of
$\alpha$, rectangle (3) commutes by naturality
of composition, while rectangle (4) commutes by naturality of $r$.
Thus the outer rectangle of the diagram is also commutative,
and using now that $$m_f=(r_f \w f)\cdot ((f^* \w \varepsilon_f)\w
f^*)\cdot (\alpha_{f^*,f,f^*} \w f) \cdot \alpha^{-1}_{f^* \w f,
f^*,f},$$ we get
\begin{equation}\label{ex1}
(\lambda \w f)\cdot m_f=m_{f^*} \cdot ((\lambda \w f)\w(f^* \w f))
\end{equation}

It then follows from (\ref{ex}) that

{\begin{align}
(\lambda \w f)\cdot m_f \cdot (i_\lambda\w(f^* \w f))
&=m_f \cdot ((\lambda \w f)\w(f^* \w f))\cdot (i_\lambda\w(f^* \w f))\notag\\
&=m_f \cdot (\eta_f\w(f^* \w f))\cdot (p_\lambda\w(f^* \w f))\label{ex2}\\
&=l_{f^* \w f}\cdot (p_\lambda\w(f^* \w f))\notag
\end{align}}

Since the morphisms $\lambda, l_{f^* \w f}$ and $p_\lambda$ all are isomorphisms, one concludes that the composite
$m_f \cdot (i_\lambda\w(f^* \w f))$ is also an isomorphism and hence we have:

\begin{proposition} \label{D} Under the hypotheses above, $\mathcal{D}_f(\lambda) \in \mathfrak{I}^r_{[A,A]}(\mathcal{S}_f)$
for all $\lambda \in \emph{\textbf{Aut}}\,_{[B,\,A]}(f^*)$.
\end{proposition}

We shall need the following easy lemma:

\begin{lemma}\label{expull}In an arbitrary category, a commutative diagram $gf=yx$
with $g$ isomorphism is a pullback iff $x$ is an isomorphism.
\end{lemma}

\begin{proposition} \label{D1} The map $\mathcal{D}_f: \emph{\text{\textbf{Aut}}}\,_{[B,\,A]}(f^*) \to
\emph{\text{\textbf{Sub}}}_{[A,\,A]}(f^*\w f)$ is a homomorphism of monoids.
\end{proposition}
\begin{proof}Quite obviously, the diagram
$$
\xymatrix{ 1_A \ar[d]_{1} \ar[r]^-{\eta_f}& f^* \w
f \ar[d]^{1_{f^*} \w f=1_{f^* \w f}}\\
1_A \ar[r]_-{\eta_f}& f^* \w f}
$$ is a pullback, showing that $\mathcal{D}_f(1_{f^*})=[(1_A, \eta_f)]=1.$

Next, for any two elements $\lambda, \lambda' \in
\text{\textbf{Aut}}_{[B,\,A]}(f^*)$, consider the diagram

$$
\xymatrix@R10pt{f_\lambda \w f_{\lambda'}\ar@{}[rrd]|{(I)}
\ar[rr]^-{f_{\lambda} \w p_{\lambda'}} \ar[dd]_{f_{\lambda} \w
i_{\lambda'}}&& f_{\lambda} \w 1_A \ar[d]_{f_{\lambda}\w
\eta_f}\ar[rr]^-{p_{\lambda}\w 1_A}&& 1_A \w 1_A \ar[dd]^{1_A \w \eta_f}\\
&& f_{\lambda} \w (f^* \w f)  \ar@{}[rru]_{(IV)} \ar@{}[d]_{(V)}\ar[rrd]^{p_{\lambda}\w (f^* \w f)}&&\\
f_{\lambda} \w (f^* \w f)\ar@{}[rrdd]|{(II)}\ar[dd]_{i_{\lambda}\w
(f^* \w f)} \ar[rru]^{f_{\lambda} \w (\lambda' \w
f)}\ar[rr]_{p_{\lambda}\w (f^* \w f)}&& 1_A \w (f^* \w f) \ar@/^7pc/ [dddd]^{l_{f^* \w f}}
\ar[dd]^{\eta_f \w (f^* \w f)}
\ar[rr]_-{1_A \w (\lambda' \w f)}&& 1_A \w (f^* \w f)\ar[dddd]^{l_{f^* \w f}}\\\\
(f^* \w f)\w (f^* \w f)\ar[dd]_{m_f}\ar[rr]_{(\lambda \w f)\w (f^*
\w f)}&& (f^*
\w f)\w (f^* \w f)\ar[dd]^{m_f} \ar@{}[rr]|{(VI)}  \ar@{}[rruu]_(0.75){(VII)}&& \\\\
f^* \w f \ar@{}[rruu]|{(III)} \ar[rr]_{\lambda \w f}&& f^* \w f \ar[rr]_-{\lambda' \w
f}&& f^* \w f}$$ in which Diagrams (I) and (II) commute by (\ref{ex}),
Diagram (III) commutes by (\ref{ex1}), Diagrams (IV), (V) commute by naturality of composition,
Diagram (VII) commutes by naturality of $l$, and Diagram (VI) commutes since
$\eta_f : 1_A \to f^* \w f$ is the unit for the multiplication $m_f$.
Thus the outer diagram, which by naturality of $l$ can be
rewritten as
$$
\xymatrix @C=0.3in @R=0.4in{f_\lambda \w f_{\lambda'} \ar[r]^-{i_\lambda \w
i_{\lambda'}} \ar[d]_{p_\lambda \w p_{\lambda'}} & (f^*
\w f)\w (f^* \w f)\ar[r]^-{m_f}& f^* \w f \ar[dd]^{(\lambda
 \lambda')\w f}\\
 1_A \w 1_A \ar[d]_{{l_{1\!\!_{_A}}}={r_{1\!\!_{_A}}}}&&\\
1_A \ar[rr]_{\eta_f}&& f^* \w f \,,}$$ commutes, and since all the
2-cells $l_{1\!_A}$, $\la$, $\la'$, $p_\lambda$ and $p_{\lambda'}$ (and hence also $_{l_{1\!_A}}\!\cdot (p_\lambda \w p_{\lambda'})$
and $(\lambda \lambda')\w f$) are isomorphisms, it follows from Lemma
\ref{expull} that the diagram is a pullback. Then, in particular, the composite $m_f \cdot (i_\lambda \w
 i_{\lambda'})$ is a monomorphism, and thus
$$\mathcal{D}_f(\lambda
 \lambda')=[(f_\lambda \w f_{\lambda'}, m_f \cdot (i_\lambda \w
 i_{\lambda'}))].$$
 Moreover,
 $$[(f_{\lambda},i_{\lambda})]\cdot[(f_{\lambda'},i_{\lambda'})]=
 [(f_\lambda \w f_{\lambda'}, m_f \cdot (i_\lambda \w
 i_{\lambda'}))]$$ in $\text{\textbf{Sub}}_{[\A,\,\A]}(f^* \w f)$ by \cite[Remark 3.3]{GM}.
 Thus $$\mathcal{D}_f(\lambda
 \lambda')=[(f_{\lambda},i_{\lambda})]\cdot[(f_{\lambda'},i_{\lambda'})]=\mathcal{D}_f(\lambda) \cdot \mathcal{D}_f(\lambda'),$$
 and hence $\mathcal{D}_f$ is a homomorphism of monoids.
\end{proof}

\begin{remark} \label{D2}  Putting $\la'=\la^{-1}$ in the proof of Proposition \ref{D1}, gives that for any $\lambda \in \mathbf{Aut}\,_{[B,\,A]}(f^*)$, the $1$--cell $f_\lambda$ defined in \eqref{ex} is invertible.
\end{remark}

\begin{proposition} \label{group} The monoid structure on $\emph{\textbf{Sub}}_{[A,\,A]}(f^* \w f)$ restricts to
a group structure on $\mathfrak{I}^A_f$. Moreover,  $\mathcal{D}_f$ induces a group homomorphism $$\overline{\mathcal{D}}_f :
\emph{\text{\textbf{Aut}}}_{[B,\,A]}(f^*) \to \mathfrak{I}^A_f.$$
\end{proposition}
\begin{proof} The 2-cell $\eta_f$ is monomorphic by assumption. Since, quite obviously, $1_A$ is invertible, and
$[(1_A, \eta_f)]\in \mathfrak{I}^l_{[A,A]}(\mathcal{S}_f)$, it follows that $[(1_A, \eta_f)]\in \mathfrak{I}^A_f$.

Next, if $[(h, i_{h})], [(g, i_{g})] \in \mathfrak{I}^A_f$, then clearly $h \w g$ is invertible.  Observe that
  $\xi^l_{i_{g}}:\mathcal{S}_f \w g \to \mathcal{S}_f$ is an
  isomorphism as $[(g, i_g)] \in \mathfrak{I}^A_f \subseteq \mathfrak{I}^l_{[A,A]}(\mathcal{S}_f)$.
   Since $i_h$ is a monomorphism, we get from Remark \ref{equivalence}  the 2-cell $i_{h}\w g : h \w g \to \mathcal{S}_f \w g$ is
  a monomorphism.
 On the other hand,  $m_f \cdot (i_{h}\w i_{g})=\xi^l_{i_{g}} \cdot (i_{h}\w g),$ and it follows that the 2-cell $i_{h \w g} : =m_f \cdot (i_{h}\w i_{g})$ is monomorphic. Thus, by \cite[Remark 3.3]{GM}, ${[(h, i_{h})]\cdot [(g,i_{g})]=}[(h \w g, i_{h \w g})]$ in $\textbf{Sub}_{[A,\,A]}(f^* \w f)$ .
Moreover, $[(h \w g, i_{h \w g})]$ lies in $\mathfrak{I}^l_{[A,A]}(\mathcal{S}_f)$ by exactly the same argument as in
the proof of \cite[Proposition 3.5]{GM}. Thus $[(h \w g, i_{h \w g})]\in \mathfrak{I}^A_f$, and hence
$\mathfrak{I}^A_f$ inherits the structure of a monoid from $\textbf{Sub}_{[A,\,A]}(f^* \w f)$.  In view of Proposition \ref{lifting},
it is easy to see that if $[(h, i_{h})] \in\mathfrak{I}^A_f$,
then its two-sided inverse is $[(h^*, i_{h^*})]$, where $h^*$ the pseudo-inverse of $h$. Therefore, $\mathfrak{I}^A_f$ is in fact a group.

In light of Proposition \ref{D} and Remark \ref{D2}, it follows from Proposition \ref{left=right} that $\mathcal{D}_f(\lambda)
\in \mathfrak{I}^A_f$, for any $\lambda \in \mathbf{Aut}\,_{[B,\,A]}(f^*)$. Proposition \ref{D1} guarantees then that $\mathcal{D}_f(\lambda)$ induces a homomorphism of groups $\overline{\mathcal{D}}_f :
\mathbf{Aut}_{[B,\,A]}(f^*) \to \mathfrak{I}^A_f$.
\end{proof}

\begin{theorem}\label{exact.1} The following sequence of groups
$$1 \xr{}\emph{\text{\textbf{Aut}}}_{[A,\,A]}(1_A)
\xr{\ol{\omega}_0}\emph{\text{\textbf{Aut}}}_{[B,\,\A]}(f^*)
\xr{\ol{\mathcal{D}}_f}\mathfrak{I}^A_f$$ is exact.
\end{theorem}
\begin{proof} To say that the sequence is exact at $\text{\textbf{Aut}}_{[A,\,A]}(1_A)$ is to say that
$\overline{\omega}_0$ is injective, which is indeed the case by Proposition \ref{mono}.

To prove exactness at $\text{\textbf{Aut}}_{[B,\,A]}(f^*)$, we have to show that
$\text{Ker}(\overline{\mathcal{D}}_f)=\text{Im}(\overline{\omega}_0)$.
For any $s \in \text{\textbf{Aut}}_{[A,\,A]}(1_A)$, the diagram
\begin{equation}\label{d.exact.1}
\xymatrix{1_A \ar[rr]^{\eta_f} \ar[d]_{(l_{1\!_A})^{-1}}&& f^* \w f \ar[d]^{(l_{f^*})^{-1} \w f}\\
1_A \w 1_A \ar[d]_{s \w 1\!_A}&& (1_A \w f^*)\w f \ar[d]^{(s \w f^*)\w f}\\
1_A \w 1_A \ar[d]_{l_{1\!_A}}&& (1_A \w f^*) \w f \ar[d]^{l_{f^*} \w f}\\
1_A \ar[rr]_{\eta_f} && f^* \w f}
\end{equation}
is commutative, as can be seen easily using the naturality  of $l$ and of $\alpha$ and the fact that
$$
\xymatrix@R10pt{(1_A \w u) \w v \ar[rr]^-{\alpha_{1\!_A, u,v}} \ar[rrdd]_{l_u \w v}&& 1_A \w (u \w v) \ar[dd]^{l_{u \w v}}\\\\
&& u \w v}
$$
is a commutative diagram for all 1-cells $u,v : A \to A$ (e.g. \cite[Proposition 1.1]{JS}). Since
$s=l_{1\!_A}\cdot(s \w 1\!_A)\cdot(l_{1\!_A})^{-1}$ by naturality of $l$ and $\overline{\omega}_0(s)=l_{f^*} \w (s \w f^*) \w (l_{f^*})^{-1}$,
it follows that (\ref{d.exact.1}) may be rewritten in the form
$$
\xymatrix{1_A \ar[rr]^{\eta_f} \ar[d]_{s}&& f^* \w f \ar[d]^{\overline{\omega}_0(s) \w f}\\
1_A \ar[rr]_{\eta_f} && f^* \w f.}
$$
Since both $s$ and $\overline{\omega}_0(s)$ are invertible 2-cells, it follows from Lemma \ref{expull} that
the diagram above is a pullback, implying that
$\overline{\mathcal{D}}_f(\overline{\omega}_0(s))=[\eta_f]=1$ in $\mathfrak{I}^A_f$.
Since $s \in \text{\textbf{Aut}}_{[A,\,A]}(1_\A)$ was arbitrary,
$\text{Im}(\overline{\omega}_0) \subseteq \text{Ker}(\overline{\mathcal{D}}_f)$.

Next, if $\lambda \in \text{\textbf{Aut}}_{[B,\,A]}(f^*)$ is such that $\overline{\mathcal{D}}_f(\lambda)=1$, then there
is an automorphism $s: 1_A \to 1_A$ such that the diagram
$$
\xymatrix{1_A \ar[r]^{\eta_f} \ar[d]_{s}& f^* \w f \ar[d]^{\lambda \w f}\\
1_A \ar[r]_{\eta_f}& f^* \w f}
$$ is a pullback, implying that in the diagram
$$
\xymatrix @C=0.3in @R=0.5in{f \ar[r]^-{r^{-1}_f} & f \w 1_A \ar[r]^-{f \w \eta_f} \ar[d]_{f \w s} & f\w (f^* \w f) \ar[r]^-{f \w ( \la \w f)}& f\w (f^* \w f)\ar[r]^-{\alpha^{-1}_{f,f^*\!\!,\, f}}& (f \w f^*) \w f \ar[d]^{\e_f \w f }\\
&f \w 1_A \ar[rru]_{f \w \eta_f}\ar[r]_-{r_f}&f \ar[rr]_-{l_f^{-1}}&& 1_A \w f}
$$ the triangle commutes, while the trapezoid commutes by (\ref{adj1}).
It then follows that the mate of $\la$ under the adjunction $f \dashv f^*$, which is the composite
$$l_f \cdot (\varepsilon_f \w f ) \cdot \alpha^{-1}_{f,f^*\!\!,\, f} \cdot (f \w ( \la \w f)) \cdot (f \w \eta_f) \cdot r^{-1}_f,$$
is in fact equal to the composite $r_f \cdot (f \w s)\cdot r^{-1}_f$. Direct inspection using
the fact that the diagram
$$
\xymatrix @C=0.3in @R=0.3in{(f\w 1_A)\w f^* \ar[rr]^-{\alpha_{f, 1_A, f^*}}\ar[dr]_{r_f \w f^*}&& f\w (1_A\w f^*)\ar[dl]^{f \w l_{f^*}}\\
&f\w f^*}
$$ commutes, shows that the mate of the last  composite
under the adjunction $f \dashv f^*$ is just $\overline{\omega}_0(s)=l_{f^*} \cdot (s \w f^*) \cdot (l_{f^*})^{-1} .$ This proves that $\overline{\omega}_0(s)=\la$.
Thus $\text{Ker}(\overline{\mathcal{D}}_f) \subseteq \text{Im}(\overline{\omega}_0)$, and hence $\text{Ker}(\overline{\mathcal{D}}_f) = \text{Im}(\overline{\omega}_0)$.
\end{proof}

\subsection{An exact sequence involving the Picard group}\label{exactPic}

For any object $A$ of $\mathbb{B}$, define the \emph{Picard Group} of $A$, denoted $\text{\textbf{Pic}}(A)$, to
be the collection of isomorphism-classes $[h]$ of invertible 1-cells $h:A \to A$ with product and inverses defined by $$[h]\cdot[g]=[h
\w g]\,\,\text{and}\,\, [h]^{-1}=[h^{*}],$$ where $h^*$ is a pseudo-inverse of $h$. As easily seen, $\text{\textbf{Pic}}(A)$
is a well-defined group with identity element $[1_A]$.

\begin{proposition}\label{picar} The assignment that takes
$[(h,i_h)]\in \mathfrak{I}^A_f$ to $[h]$ defines a group homomorphism
$$\Omega_f :\mathfrak{I}^A_f \to \emph{\text{\textbf{Pic}}}(A).$$
\end{proposition}
\begin{proof} For any $[(h,i_h)]\in \mathfrak{I}^A_f$, $[h] \in \text{\textbf{Pic}}(A)$ by the very definition
of $\mathfrak{I}^A_f$. The product $[(h,i_{h})]\cdot [(h', i_{h'})]$ of $[(h,i_{h})], \, [(h',i_{h'})] \in \mathfrak{I}^A_f$
is the pair $([h\w h'],i_{h\w h'})$, where $i_{h\w h'}$ is the composite
$$i_{h\w h'}: h\w h' \xr{i_{h}\w i_{h'}}   (f^*\w f) \w (f^*\w f)\xr{m_f} f^* \w f$$ (see the proof of Proposition \ref{group}).  Therefore,  $\Omega_f$ preserves
the product, and hence is a group homomorphism.
\end{proof}

\begin{theorem}\label{exact} The sequence of groups
$$1 \xr{}\emph{\text{\textbf{Aut}}}_{[A,\,A]}(1_A)\xr{\ol{\omega}_0}\emph{\textbf{Aut}}_{[B,\,A]}(f^*)
\xr{\ol{\mathcal{D}}_f} \mathfrak{I}^A_f \xr{\Omega_f} \emph{\textbf{Pic}}(A)$$ is exact.
\end{theorem}
\begin{proof} By Theorem \ref{exact.1}, it suffices to show that the sequence is exact at $\mathfrak{I}^A_f$.
So, suppose $[(h,i_h)]\in \mathfrak{I}^A_f$ is such that
$\Omega_f([(h,i_h)])=[h]=[1_A]$. Then there exists an isomorphism
$\tau: h \to 1_A$ in $[A,A]$. Define $\lambda$ to be the
composite
$$f^* \xr{(\xi^*_{i_h})^{-1}} h \w f^* \xr{\tau \w f^*} 1_A \w f^*
\xr{l_{f^*}}f^*.$$ It is clear that $\lambda \in
\text{\textbf{Aut}}_{[B,\,A]}(f^*)$. We claim that
$\ol{\mathcal{D}}_f(\lambda)=[(h,i_h)]$. Indeed, we know that
the diagram
$$\xymatrix@R10pt{(h \w f^*)\w f \ar[rr]^-{\alpha_{h,f^*\!\!,\,f}}\ar[dd]_{(\tau \w f^*)\w
f}&& h \w (f^* \w f)\ar[dd]^{\tau \w (f^* \w f)}\\\\
(1_A \w f^*)\w f \ar[rr]_-{\alpha_{1\!_A, f^*\!\!,\,f}}&& 1_A \w (f^* \w
f)}$$ commutes by naturality of $\alpha$, and
$l_{f^* \w f}\cdot \alpha_{1\!_A, f^*\!\!,\,f}=l_{f^*} \w f$ by
  one of the two coherence axioms (see \cite[Proposition 1.1]{JS}). Since
 $(\xi^r_{i_h})^{-1}=\alpha_{h,f^*\!\!,\,f}\cdot ((\xi^*_{i_h})^{-1} \w f)$ by the dual of \cite[Remark 4.2]{GM}),
the 2-cell $\lambda \w f$ can be rewritten as follows
$$f^* \w f \xr{(\xi^r_{i_h})^{-1}}h \w (f^* \w f)\xr
{\tau \w (f^* \w f)} 1_A \w (f^* \w f)\xr{l_{f^* \w f}}f^*
\w f.$$
In the following diagram
$$\xymatrix@R10pt{h \ar@{}[rrdd]|{(1)}\ar[rr]^-{(r_h)^{-1}}\ar[dd]_{i_h}&& h \w 1_A
\ar@{}[rrdd]|{(2)}\ar[rr]^-{\tau \w 1_A}\ar[dd]^{h \w \eta_f}&&
1_A \w 1_A \ar@{}[rrdd]|{(3)} \ar[dd]^{1_A \w
\eta_f}\ar[rr]^-{{l_{1\!\!_{_A}}}={r_{_1\!\!_{_A}}}}&&1_A \ar[dd]^{\eta_f}\\\\
f^* \w f \ar[rr]_{(\xi^r_{i_h})^{-1}}&&h \w (f^* \w f)\ar[rr]_{\tau
\w (f^* \w f)}&&1_A \w (f^* \w f) \ar[rr]_{l_{f^* \w f}}&& f^* \w
f},$$ Square (2) commutes by naturality of composition, while Square
(3) commutes by naturality of $l$. We claim that Square (1) is also
commutative. Indeed, using that
\begin{equation}\label{unit}
m_f \cdot ((f^* \w f)\w \eta_f)=r_{f^* \w f}
\end{equation} since $m_f$ is the multiplication for the $A$-ring
$S_f$, we have:
\begin{align*}
&\xi^r_{i_h}\cdot (h \w \eta_f)\cdot (r_h)^{-1}=&\text{since}\,\,\xi^r_{i_h}= m_f \cdot (i_h \w (f^* \w f))\\
&=m_f \cdot (i_h \w (f^* \w f))\cdot (h \w \eta_f)\cdot (r_h)^{-1}
&\text{by naturality of composition}\\
&=m_f \cdot ((f^* \w f)\w \eta_f)\cdot (i_h \w 1_A)\cdot (r_h)^{-1}
&\text{by (\ref{unit})}\\
&=r_{f^* \w f}\cdot (i_h \w 1_A)\cdot (r_h)^{-1} &\text{by naturality of} \,\,r\\
&=i_h \cdot r_h \cdot (r_h)^{-1}\\
&=i_h.
\end{align*} Thus the diagram
$$
\xymatrix{ h \ar[r]^-{i_h} \ar[d]_{{r_{1\!\!_{_A}}}\cdot (\tau \w 1_A)
\cdot(r_h)^{-1}}& f^* \w f \ar[d]^{\lambda \w
f}\\
1_A \ar[r]_-{\eta_f}& f^* \w f }$$ commutes, and since the
composite ${r_{1\!\!_{_A}}}\cdot (\tau \w 1_A) \cdot(r_h)^{-1}$ is an
isomorphism, it follows from Lemma \ref{expull} that the diagram is
a pullback. Hence $\ol{\mathcal{D}}_f(\lambda)=[(h,i_h)]$, and thus
$\text{Ker}(\Omega_f)\subseteq \text{Im}(\ol{\mathcal{D}}_f).$

Now, if $i_h: h \to f^* \w f$ is such that there are an automorphism
$\lambda \in \text{\textbf{Aut}}_{[B,\,A]}(f^*)$ and a pullback
$$
\xymatrix{ h \ar[r]^-{i_h} \ar[d]_{p_h}& f^* \w f \ar[d]^{\lambda \w
f}\\
1_A \ar[r]_-{\eta_f}& f^* \w f \,,}$$ then clearly the 2-cell $p_h
: h \to 1_A$ is an isomorphism and thus $\Omega_f([(h,
i_h)])=[h]=[1_A]$. Thus $\text{Im}(\ol{\mathcal{D}}_f)\subseteq
\text{Ker}(\Omega_f)$ and hence $\text{Ker}(\Omega_f)=
\text{Im}(\overline{\mathcal{D}}_f).$ This completes the proof.
\end{proof}

\begin{example}\label{firmfirst}
Let $\varphi : R \to S$ be a homomorphism of firm rings as in Example \ref{Firm}. If $\varphi$ is injective, then we can apply Theorem \ref{exact} to the adjunction ${}_SS_R \dashv {}_RS_S$ in $\mathbf{Firm}$, and we get the exact sequence of groups
\begin{equation}\label{first}
\xymatrix{1 \ar[r] & \mathbf{Aut}({}_RR_R) \ar[r] & \mathbf{Aut}({}_SS_R) \ar[r]& \mathbf{Inv}_R(S) \ar[r] & \mathbf{Pic}(R),}
\end{equation}
where $\mathbf{Aut}({}_RR_R)$, (resp. $\mathbf{Aut}({}_SS_R)$) denote the group of $(R,R)$--bimodule (resp. $(S,R)$--bimodule) automorphisms of $R$ (resp. $S$),  $\mathbf{Inv}_R(S)$ is the group of invertible $R$--subbimodules of $S$, and $\mathbf{Pic}(R)$ is the Picard group of the ring $R$. The exact sequence \eqref{first} was obtained, as a generalization of the unital case \cite{Miyashita:1973}, in \cite[Proposition 1.4]{ElKaoutit/Gomez:2010} for any extension of rings with the same set of local units. Every such an extension is clearly an injective homomorphism of firm rings.
\end{example}

Given an arbitrary category $\mathbf{C}$, we write $\pi_0 (\mathbf{C})$ for the collection of the isomorphism
classes of objects of $\mathbf{C}$. For any $C \in \mathbf{C}$, $[C]$ denotes the class of $C$.  Clearly, for
any functor $S: \mathbf{C} \to \mathbf{D}$, the assignment $[C] \to [S(C)]$ yields a map
$\pi_0(S) : \pi_0(\mathbf{C}) \to \pi_0(\mathbf{D})$.

Quite obviously, the assignment $[h] \to [f\w h]$ yields a map $\text{\textbf{Pic}}(A) \xr{[f \w-]} \pi_0([A,B])$,
where $\pi_0([A,B])$ denotes the pointed set of the isomorphism
classes $[g]$ of 2-cells $g: A \to B$ with a distinguished class $[f]$.
Since $f \w 1_A \simeq f$, $[f \w-]$ is morphism of pointed sets.

\begin{theorem}\label{ex.pic} The following sequence of pointed sets
$$\mathfrak{I}^A_f\xr{\Omega_f} \emph{\textbf{Pic}}(A) \xr{[f \w-]} \pi_0([A,B])$$ is exact.
\end{theorem}
\begin{proof}Since $\mathfrak{I}^\A_f \subseteq \mathfrak{I}^l_{[A,A]}(S_f)$, it is clear that
$([f \w-]\cdot \Omega_f)([(h,i_{h})])=[f]$ for all $[(h,i_{h})]\in \mathfrak{I}^A_f$. So it remains to show
that if $[g] \in \textbf{Pic}(A)$ is such that $$[f \w-]([g])=[f\w g]= [f],$$ then there exists
$[(h,i_{h})]\in \mathfrak{I}^A_f$ with $[g]=\Omega_f([(h,i_{h})])=[h].$ Since $[f\w g]= [f]$, there is an
isomorphism $\sigma : f\w g \to f$ in $[A,B]$. It then follows from Proposition
\ref{lifting} that $[(g, i_g)]\in \mathfrak{I}^A_f$, where $i_g$ is the composite
$g \xr{\eta_f \w g}f^* \w f \w g \xr{f^* \w \sigma} f^* \w f.$ Then clearly $\Omega_f([(g,i_{g})])=[g]$.
\end{proof}

\subsection{Comonadicity}\label{comonadicity}

Recall from \cite[p. 172]{GM} that there is a map $$\Gamma_{f}:\mathfrak{I}^l_{[A,A]}(S_f) \to \textbf{End}_{B\texttt{-cor}}(\mathfrak{C}_f)$$
that takes $[(h, i_h)]\in \mathfrak{I}^l_{[A,A]}(S_f)$ to the composite
$$f \w f^* \xr{\xi^{-1}_{i\!_h}\w f^*} (f \w h) \w f^* \xr{\alpha_{f,h,f^*}}f \w (h \w f^*)
\xr{f \w \xi^*_{i\!_h}} f \w f^*.$$

\begin{proposition} \label{comonadic} Suppose that the functor $[A,f]=f \w -:[A,A]\to [A,B]$ is comonadic. Then
$\Gamma_{f}$ restricts to an isomorphism of groups
$$\overline{\Gamma}_{f}: \mathfrak{I}^A_f\to \emph{\textbf{Aut}}_{B\texttt{-cor}}(\mathfrak{C}_f).$$
\end{proposition}
\begin{proof}  The functor $[A,f]$ is precomonadic if and only if the unit of the adjunction $[A,f]\dashv [A,f^*]$ is a
componentwise monomorphism. So $\eta_f : 1 \to f^*\circ f$ is right  pure in the monoidal category $[A,A]$
(meaning that $\eta_f \w h: 1\w h \to (f^*\circ f)\w h$ is monomorphic for all 1-cells $h:A \to A$), provided the functor
$[A,f]$ is (pre)comonadic. Consequently, according to \cite[Proposition 4.4]{GM}, $\mathfrak{I}^l_{[A,A]}(S_f)$ inherits
the structure of a monoid from $\text{\textbf{Sub}}_{[A,\,A]}(f^* \w f)$. Moreover, the map $\Gamma_{f}:\mathfrak{I}^l_{[A,A]}(S_f) \to
\textbf{End}_{B\texttt{-cor}}(\mathfrak{C}_f)$ is an isomorphism of monoids by
\cite[Theorem 4.9]{GM}. If $[(h, i_h)]\in \mathfrak{I}^A_f$, then $\xi^*_{i\!_h}$ is an isomorphism and hence is so $\Gamma_{f}([(h, i_h)])$.
Thus, $\Gamma_{f}$ restricts to a monomorphism $\overline{\Gamma}_{f}: \mathfrak{I}^A_f\to \textbf{Aut}_{B\texttt{-cor}}(\mathfrak{C}_f)$
of groups. To show that $\overline{\Gamma}_{f}$ is surjective, note first that if $[(h, i_h)] \in \mathfrak{I}^l_{[A,A]}(S_f)$
is such that $\Gamma_{f}([(h,i_h)])\in \textbf{Aut}_{B\texttt{-cor}}(\mathfrak{C}_f)$, then $[(h, i_h)]\in \mathfrak{I}^r_{[A,A]}(S_f)$.
Indeed, if the composite
$$\Gamma_{f}([(h,i_h)]):f \w f^* \xr{\xi^{-1}_{i\!_h}\w f^*} (f \w h) \w f^* \xr{\alpha_{f,h,f^*}} f \w (h \w f^*)
\xr{f \w \xi^*_{i\!_h}} f \w f^*$$ is an isomorphism, then $f \w\xi^*_{i\!_h}$ is also an isomorphism. But by hypothesis
the functor $[A,f]=f \w -$ is comonadic, and in particular conservative. Hence $\xi^*_{i\!_h}$ is an isomorphism too. Thus
$[(h, i_h)]\in \mathfrak{I}^r_{[A,A]}(S_f)$. Consider now any $\alpha \in \textbf{Aut}_{B\texttt{-cor}}(\mathfrak{C}_f)$.
Then, since $\Gamma^{-1}_{f}$ is a morphism of monoids, one has the following equalities in $\mathfrak{I}^l_{[A,A]}(S_f)$:
$$\Gamma^{-1}_{f}(\alpha) \cdot\Gamma^{-1}_{f}(\alpha^{-1})=\Gamma^{-1}_{f}(\alpha \cdot \alpha^{-1})=\Gamma^{-1}_{f}(1_{\mathcal{S}_f})=[(1_A, \eta_f)].$$ Similarly,  $\Gamma^{-1}_{f}(\alpha^{-1})\cdot \Gamma^{-1}_{f}(\alpha)=[(1_A, \eta_f)].$
If now $\Gamma^{-1}_{f}(\alpha)=[(h,i_{h})]$ and $\Gamma^{-1}_{f}(\alpha^{-1})=[(h',i_{h'})]$, then since by \cite[Proposition 4.4]{GM}, the product
$[(h,i_{h})]\cdot [(h', i_{h'})]$ in $\mathfrak{I}^l_{[A,A]}(S_f)$
is the pair $([h\w h'],i_{h\w h'})$, where $i_{h\w h'}$ is the composite
$$i_{h\w h'}: h\w h' \xr{i_{h}\w i_{h'}}   (f^*\w f) \w (f^*\w f)\xr{m_f} f^* \w f,$$ it follows that
$h \w h'\simeq 1_A$ and $h' \w h \simeq 1_A$. Hence $[h] \in \textbf{Pic}(A).$ Since $\Gamma_{f}(\Gamma^{-1}_{f}(\alpha))=\alpha$ is an isomorphism,
$\Gamma^{-1}_{f}(\alpha)\in \mathfrak{I}^r_{[A,A]}(S_f)$, as we have shown above. Thus, $\Gamma^{-1}_{f}(\alpha)\in \mathfrak{I}^A_f$,
and hence $\overline{\Gamma}_{f}$ is surjective. This completes the proof.
\end{proof}

\begin{remark} \label{intersection} We have proved in passing that, when the functor $$[A,f]=f \w -:[A,A]\to [A,B]$$
is comonadic, then
$$\mathfrak{I}^A_f=\mathfrak{I}^l_{[A,A]}(\mathcal{S}_f)\cap \mathfrak{I}^r_{[A,A]}(\mathcal{S}_f).$$
\end{remark}

As a corollary, we get:
\begin{proposition} \label{equality} Whenever the functor $$[A,f]=f \w -:[A,A]\to [A,B]$$ is comonadic,
we have an equality of groups $$\mathfrak{I}^A_f=(\mathfrak{I}^l_{[A,A]}(\mathcal{S}_f))^\times,$$
where $(-)^\times$ is the functor taking a monoid to its group of invertible elements.
\end{proposition}

\begin{remark}\label{sufficientcomonadicity}
In \cite[Section 5]{GM} some sufficient conditions for the comonadicity of the functor $[A,f]=f \w -:[A,A]\to [A,B]$ are investigated. Concretely, if $[A,f]$ preserves equalizers and $\eta_f$ is right regular $A$--pure (see \cite[Definition 5.1]{GM}), then $[A,f]$ is comonadic. This generalizes the ``faithfully flat'' classical situation. The functor $[A,f]$ also becomes comonadic if $f$ is a separable $1$--cell (that is, if $\eta_f$ is a split monomorphism in the category $[A,A]$) (see \cite[Proposition 5.5]{GM}).
\end{remark}

\subsection{Duality}\label{duality}
Let $\mathbb{B}$ be a bicategory whose
hom-categories admit finite colimits and coimages and let $\eta_f, \e_f
: f \dashv f^* : B \to A$ be an adjunction in $\mathbb{B}$ such that $\e_f:f\w f^* \to 1_B$ is epimorphic in $[B,B]$.
 Let $\mathfrak{C}_f$ be the corresponding
$B$-coring.  Write $\mathcal{Q}^{B,\,l}_f$ (resp. $\mathcal{Q}^{B,r}_f$) for the subset of $\mathcal{Q}^l_{[B,B]}(\mathfrak{C}_f)$ (resp. $\mathcal{Q}^l_{[B,B]}(\mathfrak{C}_f)$) determined the elements $[(h, i_h)]$ with $h \in \text{\textbf{Pic}}(B)$.
Then $\mathcal{Q}^{B,\,l}_f=\mathcal{Q}^{B,r}_f$ and we write $\mathcal{Q}^B_f$ to denote either $\mathcal{Q}^{B,\,l}_f$
or $\mathcal{Q}^{B,\,r}_f$.

Recall that for any bicategory  $\mathbb{B}$, $\mathbb{B}^\text{co}$ is a bicategory
obtained from $\mathbb{B}$ by reversing 2-cells, i.e., $\mathbb{B}^\text{co}(A,B)=\mathbb{B}(A,B)^\text{op}$.
Applying now Theorems \ref{exact} and \ref{ex.pic} to the bicategory $\mathbb{B}^\text{co}$ gives:

\begin{theorem} \label{exact.d.}  We have an exact sequence of groups
$$1 \xr{}\emph{\text{\textbf{Aut}}}_{[B,\,B]}(1_B)\xr{\widehat{\omega}_0}\emph{\textbf{Aut}}_{[A,\,B]}(f)
\xr{\ol{\mathcal{D}}\!_{f^*}} \mathcal{Q}^B_{f} \xr{\Omega_{f^*}} \emph{\textbf{Pic}}(B),$$
and an exact sequence of pointed sets
$$\mathcal{Q}^B_f\xr{\Omega_{f^*}} \emph{\textbf{Pic}}(B) \xr{[f^* \w -]} \pi_0([B,A]).$$
Here,
\begin{itemize}
  \item $\widehat{\omega}_0(1_B \xr{s} 1_B)=f \xr{(l\!_f)^{-1}}1_B \w f \xr{s \w f} 1_B \w f \xr{l_f}f$,
  \item $\ol{\mathcal{D}}\!_{f^*}(f \xr{\sigma} f)=[(P,\pi_P)]$, where $\xymatrix{f \w f^* \ar[r]^{\e_f}\ar[d]_{\sigma \w f^*}
  &1_B \ar[d]\\ f\w f^* \ar[r]_{\pi_P}& P}$ is a pushout, and
  \item $\Omega_{f^*}([P,\pi_P])=[P]$.
\end{itemize}
\end{theorem}

When the functor $[B,f^*]$ is monadic, we have, by \cite[Theorem 4.11]{GM}, that the map
\[\overline{\Gamma}_{f^*}:\mathcal{Q}^l_{[B,B]}(\mathfrak{C}_f)\to
\text{End}_{A-ring}(\mathcal{S}_f),
\]
given by
\[
[(p, \pi_p)]\longrightarrow (f^* \w f \xrightarrow{f^* \w \xi_{\pi_p}}f^* \w (p \w f)
\xrightarrow{\alpha^{-1}_{f^*,p,f}}(f^* \w p)\w f
\xrightarrow{(\xi^*_{\pi_p})^{-1}\w f}f^* \w f
\]
is an isomorphism of monoids.

Now, the dual version of Proposition \ref{comonadic} yields

\begin{proposition}\label{monadic}
 Suppose that the functor $[B,f^*]=f^* \w -:[B,B]\to [B,A]$ is monadic. Then
$\Gamma_{f^*}$ restricts to an isomorphism of groups
$$\overline{\Gamma}_{f^*}: \mathcal{Q}^{B}_f\to \emph{\textbf{Aut}}_{A\texttt{-ring}}(\mathcal{S}_f).$$
\end{proposition}

\begin{example}\label{firmsecond}
Let $\varphi : R \to S$ be a homomorphism of firm rings as in Example \ref{Firm}. Now, the adjunction ${}_SS_R \dashv {}_RS_S$ in $\mathbf{Firm}$ leads to the functor $S \otimes_S -: \mathbf{Firm}(S,S) \to \mathbf{Firm}(S,R)$ which is monadic according to Beck's Theorem. Moreover, the isomorphism  $S \otimes_S S \cong S$ becomes an isomorphism of $R$--rings, so that, by Proposition \ref{monadic}, we get an isomorphism of groups $\mathcal{Q}^S_S \cong \textbf{Aut}_{R-\texttt{ring}}(S)$.  We thus get from Theorem \ref{exact.d.} an exact sequence of groups
\[
\xymatrix{1 \ar[r] & \mathbf{Aut}({}_SS_S) \ar[r] &  \mathbf{Aut}({}_SS_R) \ar[r] & \mathbf{Aut}_{R-ring}(S) \ar[r] & \mathbf{Pic}(S),}
\]
which generalizes \cite[Proposition 2.3]{ElKaoutit/Gomez:2010}.
\end{example}

\section{Applications}

In this section, we apply the results from Section \ref{SucesionesExactas} to an adjoint pair in a bicategory of bimodules. This bicategory is built over an abstract monoidal category subject to some requirements which, of course, are fulfilled by the category of abelian groups, recovering the usual bicategory of bimodules.  With this tool at hand, we treat the case of a homomorphism of commutative algebras. In particular, the group isomorphism involving first Amitsur cohomology and the Picard group is proved.

\subsection{The bicategory of bimodules}\label{bimodules}  Suppose that $\V=(\V, \otimes, I)$ is
a  monoidal category such that the category $\V$ admits reflexive coequalizers, and that the latter are
preserved, as in the biclosed case, for instance,  by the functors $M \otimes - , - \otimes M : \V \to \V$, for all $M \in \V$. We will briefly recall basic notions and  results about (commutative) monoids and modules over them in
monoidal categories; all can be found in \cite{M}.

 For simplicity of exposition we treat $\otimes$ as strictly
associative and $I$ as a strict unit, which is justified by Mac Lane's coherence theorem \cite{M}.

Recall that, for a $\V$--algebra $\mathbf{A} = (A,m_A,e_A)$,   a \emph{left} $\mathbf{A}$-\emph{module}
is a pair $(M, \rho_M)$, where $M$ is an object of $\V$ and $\rho_M: A \otimes M \to M$ is a morphism in $\V$,
called the \emph{action} (or the $\mathbf{A}$-\emph{action}) on $M$, such that $\rho_M(m_A \otimes M)=
\rho_M(A \otimes \rho_M)$ and $\rho_M (e_A \otimes M)=1$.

The left $\mathbf{A}$-modules are the objects of a category ${_{\mathbf{A}}\!\!\V}$.
A morphism of left $\mathbf{A}$-modules is a morphism in $\V$ of the underlying $\V$-objects that
commutes with the actions of $\mathbf{A}$. In a similar manner, one defines the category $\V_{\mathbf{A}}$ of right $\mathbf{A}$-modules.

If $\textbf{A}$ and $\textbf{B}$ are algebras in $\V$, then an $(\textbf{A}, \textbf{B})$-bimodule $M$ in $\V$ is an object of $\V$
with commuting left $\textbf{A}$-module and right $\textbf{B}$-module structures.  The category of $(\textbf{A}, \textbf{B})$-bimodules
is denoted $_\textbf{A}\!\V_\textbf{B}$.

If  $(M,\rho_M) \in \V_{\mathbf{A}}$
and $(N,\rho_N) \in {_{\mathbf{A}}\!\!\V}$, then the \emph{tensor product of} $(M,\varrho_M)$ \emph{and} $(N,\rho_N)$
\emph{over} $\mathbf{A}$ is the object part of the following (reflexive) coequalizer
$${\xymatrix{ M \!\otimes \! A \otimes \! N \ar@{->}@<0.5ex>[rr]^-{\varrho_M \otimes N} \ar@
{->}@<-0.5ex> [rr]_-{M \otimes \rho_N}&& M \! \otimes \! N \ar[r]^{q_{M,N}}& M \! \otimes_\textbf{A} \!N.}}$$

\noindent Moreover, if $M \in {_{\textbf{B}} \!\V}\!_\textbf{A} $ and
$N \in {_\textbf{A} \!\!\V}_{\textbf{C}}$, then $ M\ww_\textbf{A} N\in {_{\textbf{B}} \!\V}_{\textbf{C}}$.
It then follows, in particular, that for a fixed $\V$-algebra $\textbf{A}$, the category
${_{\textbf{A}} \!\!\V}\!_{\textbf{A}}$ of $(\textbf{A},\textbf{A})$-bimodules in $\V$ is a
(non-symmetric) monoidal category with tensor product of two $(\textbf{A},\textbf{A})$-bimodules
being their tensor product over $\mathbf{A}$ and the unit for this tensor product being the
$(\textbf{A},\textbf{A})$-bimodule  $A$.

This allows us (see, for example, \cite{BS}) to construct the bicategory $\textbf{Bim}(\V)$ in which:
\begin{itemize}
  \item Objects are $\V$-algebras,
  \item $\textbf{Bim}(\V)(\textbf{A},\textbf{B})={_{\textbf{B}} \!\V}\!_{\textbf{A}}$;
  \item 2-cells are bimodule morphisms.
\end{itemize} Although the 1-cells in a bicategory  are usually denoted using the
arrow symbols, we sometimes, as here,  find it convenient to write
$\textbf{A} \nr \textbf{B}$ instead of $\textbf{A} \to \textbf{B}$. Thus, $M:
\textbf{A}\nr \textbf{B}$ means that $M$ is a $(\textbf{B},\textbf{A})$-bimodule.

The horizontal composite $N \w M$ of $M: \textbf{A} \nr \textbf{B}$ and $N :
\textbf{B}\nr \textbf{C} $ is the $(\textbf{C},\textbf{A})$-bimodule $N \ww\,_{\textbf{B}} M$, while
the vertical composition of two 2-cells is the ordinary composition of bimodule morphisms.

We write $\textbf{I}$ for the trivial $\V$-algebra $(I, r_I=l_I: I\ww I \to I, 1_I : I \to I)$. Then, for any $\V$-algebra
$\textbf{A}$, the category $\textbf{Bim}(\V)(\textbf{I}, \,\textbf{A})$ is (isomorphic to) the category of left
$\textbf{A}$-modules  $_\textbf{A}\!\!\V$, while the category $\textbf{Bim}(\V)(\textbf{A}, \, \textbf{I})$ is
(isomorphic to) the category of right $\textbf{A}$-modules $\V\!_\textbf{A}$. Moreover,
if $M: \textbf{A} \nr \textbf{B}$ is an $(\textbf{B},\textbf{A})$-bimodule, then  the diagrams
$$\xymatrix{\textbf{Bim}(\V)(\textbf{I}, \, \textbf{A}) \ar@{=}[d]\ar[r]^-{M
\w-}& \textbf{Bim}(\V)(\textbf{I}, \, \textbf{B}) \ar@{=}[d]\\
_\textbf{A}\!\!\V \ar[r]_{M \otimes_\textbf{A} -}& _{\textbf{B}}\!\V}$$ and
$$\xymatrix{\textbf{Bim}(\V)(\textbf{B}, \, \textbf{I})
\ar@{=}[d]\ar[r]^-{- \w M}& \text{\textbf{Bim}}(\V)(\textbf{A}, \,\textbf{I}) \ar@{=}[d]\\
\V_\textbf{B} \ar[r]_{-\otimes_{\textbf{B}} M}& \V_\textbf{A}\,,}$$ where the vertical morphisms
are the isomorphisms, are both commutative.

We henceforth suppose in addition that the category $\V$ admits, besides reflexive coequalizers,
all finite limits and image factorizations. Then, for any two $\V$-algebras $\textbf{A}$ and $\textbf{B}$,
the category ${_{\textbf{A}} \!\!\V}_{\textbf{B}} $, being the Eilenberg-Moore category for the monad
$$A \ww  -\ww \,B : \V \to\V,$$ also admits finite limits (see, for example, \cite{FB}) and image
factorizations (see \cite{A}). So we are in a position to apply Theorems \ref{exact} and \ref{ex.pic} to obtain
the following result.

\begin{theorem}\label{main.alg}Let $\V=(V, \ww, I)$ be a monoidal category with $\V$
admitting finite limits, image factorizations and reflexive coequalizers. Assume that the latter are preserved by the tensor product,
and let $\textbf{A},\textbf{B}$ be two $\V$-algebras. Then for any adjunction $\eta_M, \e_M:M \dashv M^*:\textbf{B}\nr \textbf{A}$
in $\emph{\textbf{Bim}}(\V)$ with monomorphic $\eta_M:A \to M^*\ww_\textbf{B} M=M^*\w M$, the
following sequence of groups
$$1 \xr{} \emph{\textbf{Aut}}_{{_{\textbf{A}}\!\V}\!_\textbf{A}}(A) \xr{\varpi_0}
\emph{\textbf{Aut}}_{{_{\textbf{A}}\!\V}_\textbf{B}}(M^*)\xr{\ol{\mathcal{D}}_M}
\mathfrak{I}^\textbf{A}_M\xr{\Omega_M}
\emph{\textbf{Pic}}(\textbf{A})$$ is exact. Moreover, the following sequence of pointed sets
$$\mathfrak{I}^A_M\xr{\Omega_M}\emph{\textbf{Pic}}(\textbf{A}) \xr{[M \ww_\textbf{A}-]} \pi_0({_{\textbf{B}}\!\V}\!_\textbf{A})$$ is exact.
\end{theorem}

\begin{example}\label{AB}
Each  morphism of $\V$--algebras $\iota:\textbf{A} \to \textbf{B}$ leads to two bimodules
$B_\iota:\textbf{A}\nr \textbf{B}$ and  $B^\iota:\textbf{B}\nr \textbf{A}$
which are both equal to $B$ as objects of $\V$ but with the bimodule structures defined by
$$(B\otimes B \xr{m_B}B, B\ww A \xr{B \ww \iota}B\otimes B \xr{m_B}B)$$ and
$$(A\ww B \xr{\iota \ww B}B\otimes B \xr{m_B}B,B\otimes B \xr{m_B}B).$$
In fact $B^\iota$ is right adjoint for $B_\iota$ in $\textbf{Bim}(\V)$ with
$$A \xr{\iota}B \simeq B\ww_\textbf{B} B=B^\iota \w B_\iota$$
as unit and $$B_\iota \w B^\iota=B\ww_\textbf{A}B \xr{\overline{m}_B}B$$ as counit.
Here $B\ww_\textbf{A}B \xr{\ol{m}_B}B$ is the unique morphism making the triangle
\[\xymatrix{B \ww B \ar[r]^-{q_{B,B}}\ar[rd]_{m_B}& B\ww_\textbf{A}B \ar[d]^{\ol{m}_B}\\
&B}\] commute.

It therefore follows that every morphism $\iota: \mathbf{A} \to \mathbf{B}$ of $\V$-algebras gives rise to two functors:
\begin{itemize}
  \item the \emph{forgetful functor} $\iota_*=B^\iota \w-: {_{\mathbf{B}}\!\V} \to {_{\mathbf{A}}\!\!\V},$ where
  for any left $\textbf{B}$-module $(M, \varrho_M)$,  $\iota_*(M, \varrho_M)$ is a left $\textbf{A}$-module via the action
$$A \otimes M \xr{\iota \otimes M} B \otimes M \xr{\varrho_M} M;$$
  \item the \emph{change-of-base functor} $\iota^*=B_\iota\w-: {_{\mathbf{A}}\!\V} \to {_{\mathbf{B}}\!\V},$ where for any
  (left) $\textbf{A}$-module $(N, \rho_N)$, $\iota^*(N, \rho_N)=B \otimes_A N$ and $B \otimes_A N$ is a (left) $\textbf{B}$-module
  via the action $$B \otimes B \otimes_A N \xr{m_B \otimes_A N} B \otimes_A N.$$
\end{itemize}
Since $B^\iota$ is right adjoint to $B_\iota$ in $\textbf{Bim}(\V)$, it follows that the forgetful functor $\iota_*$ is
right adjoint to the change-of-base functor functor $\iota^*$.

If we specialize Theorem \ref{main.alg} to the adjunction $B_\iota \dashv B^\iota$ in $\textbf{Bim}(\V)$,
we obtain the following exact sequence of groups:
\begin{equation}\label{seq.0}1 \xr{} \textbf{Aut}_{{_{\textbf{A}}\!\V}\!_\textbf{A}}(A) \xr{\varpi_0}
\textbf{Aut}_{{_{\textbf{A}}\!\V}_\textbf{B}}(B^\iota)\xr{\ol{\mathcal{D}}_{B_\iota}}
\mathfrak{I}^\textbf{A}_{B_\iota}\xr{\Omega_{B_\iota}} \textbf{Pic}(\textbf{A})\end{equation}
and the following exact sequence of pointed sets:
\begin{equation}\label{seq.1}
\mathfrak{I}^\textbf{A}_{B_\iota}\xr{\Omega_M}\textbf{Pic}(\textbf{A}) \xr{[B_\iota \ww_\textbf{A}-]} \pi_0({_{\textbf{B}}\!\V}\!_\textbf{A})
\end{equation}
\end{example}

Let now assume that $\V$ is a symmetric monoidal category with symmetry $\tau$. Recall that a $\V$-algebra is called \emph{commutative} if the multiplication map is unchanged when composed with the symmetry.

Given a morphism $\iota:\textbf{A} \to \textbf{B}$ of commutative $\V$-algebras, consider the
associated adjunction $B_\iota \dashv B^\iota$ in $\textbf{Bim}(\V)$. Write $\mathcal{S}_\iota$ for $\mathcal{S}_{\textbf{B}_\iota}$.
Then  $\mathcal{S}_\iota=B^\iota \ww_\textbf{B} B_\iota \simeq \textbf{B}$, where the left and right actions of
$\textbf{A}$ on $\textbf{B}$ are given by the compositions
$$\rho_l:A\ww B \xr{\iota \ww B}B \otimes B \xr{m_B}B \,\,\, \text{and}\,\,\, \rho_r:B\ww A \xr{B \ww \iota}B \otimes B \xr{m_B}B$$
respectively. Since $\iota$ is a morphism of commutative $\V$-algebras, these actions coincide (in the sense that
$\rho_r=\rho_l \cdot \tau_{B,A}$), and one concludes that
$\textbf{Sub}_{{_{\textbf{A}}\!\V}\!_\textbf{A}}(\mathcal{S}_{\iota})=\textbf{Sub}_{{_{\textbf{A}}\!\V}}(\mathcal{S}_{\iota})=
\textbf{Sub}_{\V\!_\textbf{A}}(\mathcal{S}_{\iota})$. Therefore
$\mathfrak{I}^l_{{_{\textbf{A}}\!\V}\!_\textbf{A}}(\mathcal{S}_\iota)=\mathfrak{I}^l_{{_{\textbf{A}}\!\V}}(\mathcal{S}_\iota)$.

Similarly, write $\mathfrak{C}_\iota$ for $\mathfrak{C}_{\textbf{B}_\iota}$. Then $\mathfrak{C}_\iota$ is the $(\textbf{B},\textbf{B})$-
bimodule $(B\ww_\textbf{A}B, m_B\ww_\textbf{A}B, B\ww_\textbf{A}m_B)$ equipped with
the coproduct $$B\ww_\textbf{A}e_\textbf{B}\ww_\textbf{A}B:B\ww_\textbf{A}B \to (B\ww_\textbf{A}B)\ww_\textbf{B}(B\ww_\textbf{A}B)\simeq B\ww_\textbf{A}B\ww_\textbf{A}B$$ and counit $m_B:B\ww_\textbf{A}B \to B$.

The unit  $e$ of $\textbf{A}$ can be seen as a morphism of commutative
$\V$-algebras $\textbf{I} \to \textbf{A}$. If $e$ is a monomorphism, using that ${_{\textbf{I}}\!\V}\!_\textbf{A}={\V}\!_\textbf{A}$
and that ${_{\textbf{A}}\!\!\V}_\textbf{I}={_{\textbf{A}}\!\!\V}$, we get from (\ref{seq.0}) and (\ref{seq.1})
the following exact sequences of groups
\begin{equation}\label{seq.2}
1 \xr{} \textbf{Aut}_{\V}(I) \xr{\varpi_0}
\textbf{Aut}_{\V_\textbf{A}}(A)\xr{\ol{\mathcal{D}}\!_A}
\mathfrak{I}^\textbf{I}_{A}\xr{\Omega_I} \textbf{Pic}(\textbf{I})
\end{equation}
and of pointed sets
\begin{equation}\label{seq.3}
\mathfrak{I}^\textbf{I}_{A}\xr{\Omega_A}\textbf{Pic}(\textbf{I}) \xr{[A \ww -]} \pi_0(_\textbf{A}\!\!\V).
\end{equation}

It is easy to see that $\mathfrak{I}^l_{\V}(\mathcal{S}_e)=\mathfrak{I}^l_{\V}(A)$ and that
$\mathfrak{I}^r_{\V}(\mathcal{S}_e)=\mathfrak{I}^r_{\V}(A)$.

\begin{proposition}\label{mon.ab} Let $\textbf{A}$ be a commutative $\V$-algebra with monomorphic unit $e:I \to A.$
Then $\mathfrak{I}^l_{\V}(A)$ is a commutative monoid, while  $\mathfrak{I}^\textbf{I}_{A}$ is an abelian group.
\end{proposition}
\begin{proof} Since $\V$ is symmetric, the monoid structure on $\mathfrak{I}^l_{\V}(A)$ is easily seen to be
commutative. This implies -- since by Proposition \ref{group} the monoid structure on $\mathfrak{I}^l_{\V}(A)$
restricts to the group structure on $\mathfrak{I}^\textbf{I}_{A}$ -- that the group $\mathfrak{I}^\textbf{I}_{A}$
is abelian.
\end{proof}

\begin{lemma}\label{l=r} For
any commutative $\V$-algebra $\textbf{A}$,  $\mathfrak{I}^l_{\V}(A)=\mathfrak{I}^r_{\V}(A)$.
\end{lemma}
\begin{proof}For any subject  $i\!_J: J\to A$ of $A$, consider the diagram
\[\xymatrix @C=0.7in @R=0.3in{ A \ww J \ar[d]_{\tau_{A,J}} \ar[r]^-{A\ww i\!_J} & A\ww A \ar[d]^{\tau_{A,A}}\\
J \ww A \ar@{-->}[rd]_{\xi^r_{i\!_J}}\ar[r]^-{i_J\ww A} & A\ww A \ar[d]^{m_A}\\
&A}\] in which the rectangle commutes by naturality of $\tau$. Since $\textbf{A}$ is commutative, $m_A \cdot \tau_{A,A}=m_A$,
and hence $\xi^r_{i\!_J}\cdot \tau_{A,J}=m_A \cdot \tau_{A,A} \cdot (A\ww i_J)=m_A  \cdot (A\ww i_J)=\xi^l_{i\!_J}$.
Thus $\xi^r_{i\!_J}\cdot \tau_{A,J}=\xi^l_{i\!_J}$ and hence $\xi^r_{i\!_J}$ is an isomorphism
(i.e. $[(i\!_J, J)]\in \mathfrak{I}^r_{\V}(A)$)) iff $\xi^l_{i\!_J}$ is so (i.e. $[(i\!_J, J)]\in \mathfrak{I}^l_{\V}(A)$).
Therefore, $\mathfrak{I}^l_{\V}(A)=\mathfrak{I}^r_{\V}(A)$.
\end{proof}

\begin{proposition}\label{end=aut} Let $\mathbf{A}$ be a commutative $\V$-algebra such that the functor
$A \ww -:\V \to {_{\mathbf{A}}\!\V}$ is comonadic. Then
$$\mathbf{End}_{\mathbf{A}\texttt{-cor}}(\mathfrak{C}_e)=\mathbf{Aut}_{\mathbf{A}\texttt{-cor}}(\mathfrak{C}_e).$$
\end{proposition}
\begin{proof} Since the functor $A \ww -:\V \to {_{\mathbf{A}}\!\V}$ is assumed to be  comonadic,
the map $$\Gamma\!_{A}:\mathfrak{I}^l_{\V}(A)=\mathfrak{I}^l_{\V}(\mathcal{S}_e) \to \mathbf{End}_{\mathbf{A}\texttt{-cor}}(\mathfrak{C}_e)$$
is an isomorphism of monoids by \cite[Theorem 4.9]{GM}. But since
\begin{itemize}
  \item the monoid isomorphism $\Gamma\!_{A}:\mathfrak{I}^l_{\V}(A) \to
\mathbf{End}_{\mathbf{A}-\texttt{cor}}(\mathfrak{C}_e)$ restricts to an isomorphism
$\ol{\Gamma}\!_{A}: \mathfrak{I}_A^\textbf{I}\to \mathbf{Aut}_{\mathbf{A}-\texttt{cor}}(\mathfrak{C}_e)$
of groups by Proposition \ref{comonadic}, and
  \item $\mathfrak{I}^\mathbf{I}_A=\mathfrak{I}^l_{\V}(A)\cap \mathfrak{I}^r_{\V}(A)=\mathfrak{I}^l_{\V}(A)$
  by Remark\ref{intersection} and by Lemma \ref{l=r},
\end{itemize} it follows that $\Gamma\!_{A}=\ol{\Gamma}\!_{A}$, and hence $\mathbf{Aut}_{\mathbf{A}\texttt{-cor}}(\mathcal{S}_e)=\mathbf{End}_{\mathbf{A}\texttt{-cor}}(\mathfrak{C}_e)$.
\end{proof}

\begin{remark} \label{commutative.1} Since for any commutative $\V$-algebra $\textbf{A}$, ${_{\mathbf A} \!\V}$
is a symmetric monoidal category with tensor product $-\ww_\textbf{A}-$ and monoidal unit $(A,m_A)$, and since the monoid of endomorphisms
of the monoidal unit of any monoidal category is commutative (e.g., (\cite[ 1.3.3.1]{SR})), it follows that
$\textbf{End}_{_\textbf{A}\!\V}(A,m_A)$ is a commutative monoid, and $\textbf{Aut}_{_\textbf{A}\!\V}(A)$
is an abelian group.
Since $\textbf{A}=\textbf{A}^{\text{op}}$ for any commutative $\V$-algebra $\textbf{A}$,
it follows that ${_{\textbf{A}}\!\!\V}\!_\textbf{A}={_{\textbf{A}\ww\textbf{A}}}\!\!\V$; and since $\textbf{A}\ww\textbf{A}$ is
again a commutative $\V$-algebra, we get that $\textbf{End}_{{_{\textbf{A}}\!\V}\!_\textbf{A}}(A\ww A)$
is a commutative monoid. Then the inclusions
$$\textbf{Aut}_{\textbf{A}\texttt{-cor}}(\mathfrak{C}_e)\subseteq \textbf{End}_{\textbf{A}\texttt{-cor}}(\mathfrak{C}_e)\subseteq  \textbf{End}_{{_{\textbf{A}}\!\!\V}\!_\textbf{A}}(A\ww A)$$ imply that $\textbf{End}_{\textbf{A}\texttt{-cor}}(\mathfrak{C}_e)$
is a commutative monoid, and that $\textbf{Aut}_{\textbf{A}\texttt{-cor}}(\mathfrak{C}_e)$ is an abelian group.
\end{remark}

\subsection{Amitsur cohomology and Picard group}\label{Picard} We still assume that $\V$ is symmetric with symmetry $\tau$, and also that $\V$ admits reflexive coequalizers that are preserved by the tensor product, and all finite limits and image factorizations.

For a commutative algebra $\mathbf{A} = (A,m,e)$ in $\V$, write $\textbf{Pic}^{c}(\textbf{A})$ for the subgroup of $\textbf{Pic}(\textbf{A})$ consisting of all classes of invertible
$(\textbf{A},\textbf{A})$-bimodules $(M, \rho_l:A\ww M \to M, \rho_r:M\ww A \to M)$ such that
$\rho_r=\tau_{A,M}\cdot \rho_l$. Then $\textbf{Pic}^{c}(\textbf{A})$ is easily seen to be an abelian group. Moreover,
given a morphism $\iota : \textbf{A}\to \textbf{B}$ of commutative $\V$-algebras,
$$\textbf{Pic}^{c}(\iota):\textbf{Pic}^{c}(\textbf{A})\to \textbf{Pic}^{c}(\textbf{B})$$
defined  by $\textbf{Pic}^{c}(\iota)([P])=[B\ww_\textbf{A} P]$, is a homomorphism of abelian groups.

It is clear that $\textbf{Pic}^{c}(\textbf{I})=\textbf{Pic}(\textbf{I})$.
It is also clear that $[A \ww -]$  factors through $\textbf{Pic}^{c}(e):\textbf{Pic}^{c}(\textbf{I}) \to \textbf{Pic}^{c}(\textbf{A})$,
i.e. the following diagram
\[\xymatrix{ \textbf{Pic}^{c}(\textbf{I}) \ar[d]_{\textbf{Pic}^{c}(e)} \ar[r]^-{[A \ww -]}& \pi_0(_\textbf{A}\!\!\V)\\
\textbf{Pic}^{c}(\textbf{A}) \ar@{^{(}->}[ru]&}\,,\] where the unlabeled
morphism is the canonical embedding, is commutative. It then follows -- since (\ref{seq.3}) is an exact sequence of pointed sets -- that
\begin{equation}\label{seq.4}
\mathfrak{I}^\textbf{I}_{A}\xr{\Omega_A}\textbf{Pic}^{c}(\textbf{I}) \xr{\textbf{Pic}^{c}(e)} \textbf{Pic}^{c}(\textbf{A})
\end{equation} is an exact sequence of abelian groups, provided $e:I \to A$ is monomorphic.

\begin{theorem}
If  $\textbf{A}$ is such that the functor $A \ww -:\V \to {_{\mathbf{A}}\!\V}$
is comonadic, then there exists an exact sequence of abelian groups
\begin{equation}\label{seq.5}
0 \xr{} \mathbf{Aut}_{\V}(\mathbf I) \xr{\varpi_0} \mathbf{Aut}_{\V_\mathbf{A}}(A)\xr{\kappa_A}
\mathbf{Aut}_{\mathbf{A}-\texttt{cor}}(\mathfrak{C}_e) \xr{o_A} \mathbf{Pic}^{c}(\mathbf{I}) \xr{\mathbf{Pic}^{c}(e)} \mathbf{Pic}^{c}(\mathbf{A}).
\end{equation}
\end{theorem}
\begin{proof}
 By Proposition \ref{comonadic}, the isomorphism $\Gamma_{A}:\mathfrak{I}^l_{\V}(A) \to
\textbf{End}_{\textbf{A}\texttt{-cor}}(\mathfrak{C}_e)$ of monoids restricts to an isomorphism
$$\ol{\Gamma}\!_{A}: \mathfrak{I}_A^\textbf{I}\to \textbf{Aut}_{\textbf{A}\texttt{-cor}}(\mathfrak{C}_e).$$
Write $\kappa_A$ for $\ol{\Gamma}\!\!_{A} \ol{\mathcal{D}}\!_A$ and write
$o_A$ for $\Omega\!_A (\ol{\Gamma}\!\!_{A})^{-1}$. Then by combining (\ref{seq.2}) with (\ref{seq.3}), one obtains
the exact sequence \eqref{seq.5}.
\end{proof}

As an immediate consequence we deduce:

\begin{proposition}\label{ker.pic} Suppose that $\textbf{A}=(A,m,e)$ is a commutative $\V$-algebra such that that the functor
$A \ww -:\V \to {_{\textbf{A}}\!\V}$ is comonadic. Then one has an isomorphism of groups
$$\emph{\textbf{Coker}}(\kappa_A) \simeq \emph{\textbf{Ker}}(\emph{\textbf{Pic}}^{c}(e)),$$
where $\emph{\textbf{Coker}}(\kappa_A)$ is the cokernel of the homomorphism $\kappa_A$.
\end{proposition}

We will need the following descrpition of $\kappa_A$:

\begin{proposition}\label{lambda} In the circumstances above, $\kappa_A(\lambda)=(A \ww \lambda^{-1})\cdot
(\lambda \ww A)\,\,\, \text{for every } \lambda \in \emph{\textbf{Aut}}_{\V_\mathbf{A}}(A).$
\end{proposition}
\begin{proof} Recall first that for any $[(J, i_{J})]\in \mathfrak{I}_A^\textbf{I}$, $\ol{\Gamma}\!_{A}([(J, i_{J})])$
is the composite
$$ A\ww A \xr{(\xi^l_{i\!_{J}})^{-1}\ww A} A\ww J \ww A \xr{A\ww \xi^r_{i\!_{J}}}A\ww A.$$

Now, take any $\lambda \in \textbf{Aut}_{\V_\textbf{A}}(A)$ and form the pullback
\[\xymatrix{ A_\lambda \ar[d]_{p_\lambda} \ar[r]^-{i_\lambda}& A \ar[d]^{\lambda}\\
I \ar[r]_-{e}& A\,.}\] Then  $\ol{\mathcal{D}}\!_A(\lambda)=[(A_\lambda, i_\lambda)]$.
Moreover, $m \cdot (i_\lambda \ww A)=\lambda^{-1} \cdot (p_\lambda \ww A)$ by (\ref{ex2}). Thus
\begin{equation}\label{left}
\xi^r_{i_\lambda}=\lambda^{-1} \cdot (p_\lambda \ww A).
\end{equation}
Then, since $\xi^l_{i_\lambda}=\xi^r_{i_\lambda}\cdot \tau_{A,\,A_\lambda}$, it follows that
\begin{equation}\label{right}
\xi^l_{i_\lambda}=\lambda^{-1} \cdot(A \ww p_\lambda).
\end{equation}

Considering now the diagram
\[\xymatrix @C=0.6in @R=0.5in{A \ww A \ar@/^3pc/ [rr]^{(\xi^l_{i\!_{\lambda}})^{-1}\ww A} \ar[r]^-{A\w p^{-1}_\lambda
\ww A}\ar@{=}[rd]& A \ww J \ww A \ar[r]^-{\lambda \ww J \ww A}
\ar[d]|{A\ww p_\lambda \ww A}& A \ww J \ww A\ar[d]|{A\ww p_\lambda \ww A} \ar@/^3pc/ [dd]^{A\ww \xi^r_{i\!_{\lambda}}}\\
& A \ww A \ar[r]_{\lambda \ww A} & A \ww A \ar[d]|{A \ww \lambda^{-1}}\\
&& A \ww A}\] in which the rectangle commutes by naturality of composition, and using (\ref{left}) and (\ref{right}),
one concludes that $\kappa_A(\lambda)=(A \ww \lambda^{-1})\cdot (\lambda \ww A).$
\end{proof}

Our next objective is to prove that the group $\textbf{Ker}(\textbf{Pic}^{c}(e))$ is isomorphic to a suitable Amitsur cohomology group $\mathcal{H}^1(e\,, \ul{\textbf{Aut}}\,_I^{\mathfrak {Alg}})$. In order to describe this cohomology group, and to prove the existence of the aforementioned isomorphism, we need to use some classical results which, for the convenience of the reader, are recalled in the Appendix. Let us first describe the functor $\ul{\mathbf{Aut}}_I^{\mathfrak{Alg}}$, which is a particular case of the given at the beginning of the Appendix.

Let
$\mathcal{E}$ be the opposite of the category of commutative $\V$-algebras, $\verb"CAlg"(\V)$.
It is well-known (e.g., \cite[Corollary C.1.1.9]{Jon}) that, under our assumptions on $\V$,  $\mathcal{E}$ has pullbacks and they are constructed
as tensor products. It is routine to check that the assignments
$$\textbf{A}\longmapsto {_{\mathbf{A}}\!\V} \quad \text{and}\quad  \textbf{A} \xr{\iota} \textbf{B}\longmapsto {_{\mathbf{A}}\!\V} \xr{\iota^*} {_{\mathbf{B}}\!\V},$$ where $\iota^*:{_{\mathbf A}\!\V} \to {_{\mathbf B}\!\V}$ is the \emph{change--of--base} functor
induced by $\iota$ (see Example \ref{AB}), give rise to an $\mathcal{E}$-indexed category (see the Appendix)
$$\mathfrak {Alg}: \mathcal{E}^{\text{op}} \to \mathbf{CAT}.$$

Since for any morphism $\iota: \textbf{B} \to \textbf{A}$ in $\mathcal{E}$ (i.e., a morphism
$\iota: \textbf{A} \to \textbf{B}$ of commutative $\V$-algebras), the functor
$\iota^*:{_{\mathbf A}\!\V} \to {_{\mathbf B}\!\V}$ admits as a right adjoint the
forgetful functor $\iota_*: {_{\mathbf B}\!\V} \to {_{\mathbf A}\!\V}$ (see Example \ref{AB}),
the $\mathcal{E}$-indexed category $\mathfrak {Alg}$ has products (in the sense of Definition \ref{havingproducts}) if and only if,
for any morphism $\kappa:\textbf{C} \to \textbf{A}$ in  $\mathcal{E}$, there is
an isomorphism $q_*p^* \to \iota^*\kappa_*$ of functors, where $ p=\iota \ww_\textbf{A} C  : C \to B \ww_\textbf{A} C$
and $ q=B \ww_\textbf{A} \kappa  : B \to B \ww_\textbf{A} C$. It is easy to see that this condition
is equivalent to saying that for any $V \in {_{\mathbf{C}}\!\V}$,
one has an isomorphism $$(B \ww_A C)\ww_{\textbf{C}}V \simeq B  \ww_\textbf{A}\!V,$$ and this is certainly the case,
since the tensor product in $\V$ preserves reflexive coequalizers
by our assumption on $\V$. Thus, $\mathfrak {Alg}$ admits products.

\begin{lemma}\label{autalg} The functor $\ul{\mathbf{Aut}}_A^{\VV} : (\mathcal{E}\!\downarrow \!\mathbf A)^{\text{op}}=\mathbf A \!\downarrow \!\verb"CAlg"(\V) \to \mathbf{Group}$ is described on objects as
\begin{equation}\label{a.1}
\ul{\mathbf{Aut}}_{\,A}^{\VV}(\mathbf{A}\xr{\iota} \mathbf{B})=\mathbf{Aut}_{{_{\mathbf B} \!\V}}(B)
\end{equation}
and on morphisms as
\begin{equation}\label{a.2}
\ul{\mathbf{Aut}}_{\,A}^{\VV}(f:(\mathbf{B}',\iota') \to (\mathbf{B},\iota))(\sigma)=m_{B'}\cdot(B' \ww f) \cdot(B' \ww \sigma)\cdot(B' \ww \,e_B)
\end{equation}
for all $\sigma \in \mathbf{Aut}_{{_{\mathbf B} \!\V}}(B)$.
\end{lemma}
\begin{proof}
Fix a commutative $\V$-algebra $\textbf{A}$, and a left $\mathbf A$-module $M_0$. We have the functor
$$\ul{\textbf{Aut}}^{\VV}_{\,M_0}:(\mathcal{E}\!\downarrow \!\mathbf A)^{\text{op}}=\mathbf A \!\downarrow \!\verb"CAlg"(\V) \to \textbf{Group}$$
sending  each object $\iota: \textbf{B}\to \textbf{A}$ of $\mathcal{E}\!\downarrow \!\mathbf A$ (i.e. a morphism
$\iota: \textbf{A}\to \textbf{B}$ of commutative $\V$-algebras) to the group
$$\ul{\textbf{Aut}}_{\,M_0}^{\VV}(\iota)=\textbf{Aut}_{{_{\mathbf B} \!\V}}(\iota^*(M_0)).$$
Note that since $\iota^*(M_0)=(B \ww_\textbf{A}M_0, m\ww_\textbf{A}M_0), $ $\ul{\textbf{Aut}}_{M_0}^{\mathfrak {Alg}}(\iota)=
\textbf{Aut}_{{_{\mathbf B} \!\V}}(B \ww_\textbf{A}M_0)$, where $B \ww_\textbf{A}M_0$ is a left $\textbf{B}$-module
via $m_B \ww_\textbf{A}M_0:B \ww B \ww_\textbf{A} M_0 \to B \ww_\textbf{A}M_0$.

Let us describe explicitly the action of $\ul{\textbf{Aut}}_{\,M_0}^{\VV}$ on morphisms. If  $f:(\textbf{B}',\iota') \to (\textbf{B},\iota)$ is a morphism  in
$(\mathcal{E}\!\downarrow \!\mathbf A)^{\text{op}}$
making the triangle \[\xymatrix{& \textbf{A} \ar[ld]_\iota \ar[rd]^{\iota'}&\\ \textbf{B} \ar[rr]_f&& \textbf{B}'  }\]  commute,
then $$\ul{\textbf{Aut}}_{M_0}^{\mathfrak {Alg}}(f):\ul{\textbf{Aut}}_{M_0}^{\mathfrak {Alg}}(\iota)\to \ul{\textbf{Aut}}_{M_0}^{\mathfrak {Alg}}(\iota')$$ takes
$\sigma \in \ul{\textbf{Aut}}_{M_0}^{\VV}(\iota)$ to the composite
$$B' \ww_\textbf{A}M_0\simeq B'\ww_\textbf{B}(B\ww_\textbf{A}M_0) \xr{B'\ww_\textbf{B} \sigma}
B'\ww_\textbf{B}(B\ww_\textbf{A}M_0)\simeq B' \ww_\textbf{A}M_0.$$

Since the following split coequalizer diagram
\[\xymatrix{ B' \ww  B \ww  B \ww_\textbf{A}M_0 \ar@{->}@<0.5ex>[rr]^-{B'\ww m_B \ww_\textbf{A}M_0} \ar@
{->}@<-0.5ex> [rr]_-{a}&& B'  \ww  B \ww_\textbf{A}M_0 \ar@/_3pc/
[ll]_{B' \ww B \ww e\!_B \ww_\textbf{A}M_0}
\ar[rr]_{B' \ww f \ww_\textbf{A}M_0}&& B'\ww B'\ww_\textbf{A}M_0 \ar[rr]_{m_{B'}\ww_\textbf{A}M_0} &&B' \ww_\textbf{A}M_0\ar@/_3pc/ [llll]_{B' \ww e\!_B\ww_\textbf{A}M_0},}\] where $a$ is the composite $(m_{B'} \ww B \ww_\textbf{A}M_0)\cdot (B' \ww f \ww B \ww_\textbf{A}M_0)$,
is the defining coequalizer for  $B'\ww_\textbf{B} (B\ww_\textbf{A}M_0)$, it follows that
$\ul{\textbf{Aut}}_{M_0}^{\mathfrak {Alg}}(f)(\sigma)$ is the unique (iso)morphism $B'\ww_\textbf{A}M_0  \to B'\ww_\textbf{A}M_0 $ making the diagram
\[\xymatrix @C=0.4in @R=0.3in{ B' \ww B  \ww_\textbf{A}M_0\ar[d]_{B' \ww f  \ww_\textbf{A}M_0}\ar[rr]^{B'\ww \sigma} &&
B' \ww B  \ww_\textbf{A}M_0\ar[d]^{B' \ww f  \ww_\textbf{A}M_0}\\
B' \ww B'  \ww_\textbf{A}M_0 \ar[d]_{m_{B'}\ww_\textbf{A}M_0}&&  B' \ww B'  \ww_\textbf{A}M_0 \ar[d]^{m_{B'}\ww_\textbf{A}M_0}\\
B'  \ww_\textbf{A}M_0 \ar[rr]_{\ul{\textbf{Aut}}_{M_0}^{\mathfrak {Alg}}(f)(\sigma)} && B'  \ww_\textbf{A}M_0}\]
commute. But since $(m_{B'}\ww_\textbf{A}M_0)\cdot (B' \ww f \ww_\textbf{A}M_0)\cdot(B' \ww e_B \ww_\textbf{A}M_0)=1$, it follows that
$$\ul{\textbf{Aut}}_{M_0}^{\VV}(f)(\sigma)=(m_{B'}\ww_\textbf{A}M_0)\cdot(B' \ww f\ww_\textbf{A}M_0) \cdot(B' \ww \sigma)\cdot(B' \ww \,e_B\ww_\textbf{A}M_0).$$

For us of  interest is the case where $M_0=A$ with the left regular action of $\textbf{A}$ on $A$.
Since $\ul{\textbf{Aut}}_A^{\VV}(\textbf{A}\xr{\iota} \textbf{B})=\textbf{Aut}_{{_{\mathbf B} \!\V}}(B \ww_\textbf{A}A)$
and since the defining coequalizer diagram for $B \ww_\textbf{A}A$ is the following split one
\[\xymatrix{ B \ww  A \ww  A \ar@{->}@<0.5ex>[rr]^-{B\ww m_A} \ar@
{->}@<-0.5ex> [rr]_-{(m_B \ww A)\cdot (B \ww \iota \ww A)}&& B  \ww  A \ar@/_2pc/ [ll]_{B \ww B \ww e_A}
\ar[r]_{B \ww \iota}& B\ww B \ar[r]_{m_B} &B \ar@/_2pc/ [ll]_{B \ww e_A},}\] it follows that
the group $\ul{\textbf{Aut}}_{\,A}^{\VV}(\iota)$ is canonically isomorphic to $\textbf{Aut}_{{_{\mathbf B} \!\V}}(B)$.
\end{proof}

Note that since for each morphism $\iota:\textbf{A} \to \textbf{B}$ of commutative $\V$-algebras, $\ul{\textbf{Aut}}_A^{\mathfrak {Alg}}(\iota)$
is an abelian group by Remark \ref{commutative.1}, it follows that the functor $\ul{\textbf{Aut}}_A^{\mathfrak {Alg}}$ takes values in the category
of abelian groups.

Now consider the augmented simplicial object in $\mathcal{E} = \verb"CAlg"(\V)^{op}$
\begin{equation}\label{Amitsure} \xymatrix{(A/I)_*:\,\,I \ar[r]^-{e} & A \ar@<-1mm>[r]_-{i_1}  \ar@<1mm>[r]^-{i_0}
& A\ww A\ar@/_2pc/ [l]_{s_0}  \ar[rr]|-{i_1}\ar@<2mm>[rr]^-{i_0} \ar@<-2mm>[rr]_-{i_2}&&
A\ww A \ww A \ar@/_3pc/ [ll]_{s_0, \, s_1}\ar@{}[r]|(.65){.\,.\,.}&},
\end{equation}
associated to the morphism $e : I \to A$, which is a particular case of \eqref{Amitsurg} in the Appendix. By applying the functor $\ul{\mathbf{Aut}}_I^{\mathfrak {Alg}}$ to $(A/I)_*$, and computing cohomology, we get the first Amitsur cohomology group $\mathcal{H}^1(e\,, \ul{\mathbf{Aut}}_I^{\mathfrak {Alg}})$. The reader is referred to the Appendix for details.

\begin{theorem} \label{Amitsur.alg}Suppose that $\textbf{A}=(A,m,e)$ is a $\V$-algebra such that that the functor
$$A \ww \,-:\V \to {_{\emph{\textbf{A}}}\!\V}$$ is comonadic. Then there is a natural isomorphism
$$\mathcal{H}^1(e\,, \ul{\emph{\textbf{Aut}}}_I^{\mathfrak {Alg}})\simeq \emph{\textbf{Ker}}(\emph{\textbf{Pic}}^{c}(e)).$$
\end{theorem}
\begin{proof} According to Proposition \ref{ker.pic}, it is enough to show that there is a natural isomorphism
$\textbf{Coker}(\kappa_A) \simeq \mathcal{H}^1(e\,, \ul{\textbf{Aut}}\,_I^{\mathfrak {Alg}}).$

Write $\mathcal{G}_e$ for the comonad on ${_\textbf{A}}\!\!\V$  generated by the adjunction
$$\xymatrix{\V   \rrtwocell^{e^*=A \ww -}_{e_*=U}{'\top} &&{_\textbf{A}}\!\!\V}$$ and write
$\mathcal{G}_e\textbf{-Coalg}(A,m_A)$ the set of all $\mathcal{G}_e$-coalgebra structures on $(A,m_A) \in {_\textbf{A}}\!\!\V$. We know from \cite[Proposition 4.5]{GM} that $\mathcal{G}_e\textbf{-Coalg}(A,m_A)=\textbf{End}_{\textbf{A}\texttt{-cor}}(\mathfrak{C}_e)$ and that $\textbf{End}_{\textbf{A}\texttt{-cor}}(\mathfrak{C}_e)=\textbf{Aut}_{\textbf{A}\texttt{-cor}}(\mathfrak{C}_e)$ by Corollary \ref{end=aut}. On the other hand, $\mathrm{Des}_\mathfrak {Alg}(A,m_A)=\mathcal{Z}^1(\iota\,, \ul{\textbf{Aut}}\,_I^{\mathfrak {Alg}})$ by Proposition \ref{Grothendieck}, and $\mathrm{Des}_\mathfrak {Alg}(A,m_A)=\mathcal{G}_e\textbf{-Coalg}(A,m_A)$ by Theorem \ref{BRB}.  It follows that $\mathcal{Z}^1(e\,, \ul{\textbf{Aut}}\,_I^{\mathfrak {Alg}})=\textbf{Aut}_{\textbf{A}\texttt{-cor}}(\mathfrak{C}_e)$.

Applying the functor $\ul{\textbf{Aut}}\,_I^{\mathfrak {Alg}}$ to \eqref{Amitsure}, and using the fact that
for any commutative $\V$-algebra $\textbf{S}$, $\ul{\textbf{Aut}}\,_I^{\mathfrak {Alg}}(S)=
\textbf{Aut}_{{_{\mathbf S} \!\V}}(S)$ is an abelian group by Remark \ref{commutative.1},
we get the following  simplicial abelian group
\[\xymatrix{(A/I, \ul{\textbf{Aut}}\,_I^{\mathfrak {Alg}})_*:\textbf{Aut}_{\V}(I) \ar[rr]^-{\ul{\textbf{Aut}}\,_I^{\mathfrak {Alg}}(e)} &&
\textbf{Aut}_{{_{\mathbf A}\!\!\V}}(A)
\ar@<-2mm>[rr]_-{\ul{\textbf{Aut}}\,_I^{\mathfrak {Alg}}(i_1)}  \ar@<2mm>[rr]^-{\ul{\textbf{Aut}}\,_I^{\VV}(i_0)} &&
\textbf{Aut}_{{_{\textbf{A}\ww \textbf{A}} \!\V}}(A\ww A)\ar@/_3pc/ [ll]_{\ul{\textbf{Aut}}\,_I^{\VV}(s_0)}  {.\,.\,.}&}\]
and the corresponding complex $C(A/I, \ul{\au}\,^{\al}_I)$ of abelian groups

$$0 \xr{} \textbf{Aut}_{\V}(I) \xr{\ul{\au}\,_I^{\VV}(e)}
 \au_{{_{\mathbf A}\!\!\V}}(A)\xr{\Delta_1}
\au_{{_{\textbf{A}\ww \textbf{A}} \!\V}}(A\ww A) \xr{\Delta_2} \cdots $$
where $$\Delta_n=\prod_{i=0}^{n} {\ul{\textbf{Aut}}}\,_I^{\VV}(i_n)^{(-1)^n}, \, n\geq 1.$$

Since $i_0=A \ww \,e$, $i_1=e \ww A$ and since the multiplication in the tensor product $\V$-algebra
$\textbf{A}\ww \textbf{A}$ is given by the composite $(m_A\ww m_A)\cdot (A \ww \tau_{A,A} \ww A)$,
it follows from (\ref{a.2}) that
$${\ul{\textbf{Aut}}}\,_I^{\VV}(i_0)(\lambda)= (m_A\ww m_A)\cdot (A \ww \tau_{A,\,A} \ww A) \cdot (A\ww A \ww A \ww \,e)
\cdot (A\ww A \ww \la)\cdot (A \ww A \ww \,e)$$
and
$$\ul{\textbf{Aut}}\,_I^{\VV}(i_1)(\lambda)=(m_A\ww m_A)\cdot (A \ww \tau_{A,\,A} \ww A) \cdot (A\ww A \ww \,e \ww A)
\cdot (A\ww A \ww \la)\cdot (A \ww A \ww \,e)$$ for all $\lambda \in \au_{{_{\textbf{A}} \!\V}}(A,m_A)$.

But since in the diagram
\[\xymatrix @C=0.6in @R=0.5in{A\ww A \ar@/_2pc/ [rdd]_{1_{A\ww A}}\ar[rd]_-{A\ww e \ww A}
\ar[r]^-{A\ww A \ww e}& A\ww A \ww A \ar[d]^{A\ww \tau_{A,\,A} }\ar[r]^-{A\ww A \ww \la} & A\ww A \ww A \ar[d]^{A\ww \tau_{A,\,A} }\ar[r]^-{A\ww A \ww A \ww e}&
A\ww A \ww A\ww A\ar[d]^{A\ww \tau_{A,\,A} \ww A}\\
&A\ww A \ww A \ar[d]_{m_A\ww A}\ar[r]_{A\ww \la \ww A} &A\ww A \ww A \ar@{=}[rd]
\ar[r]_-{A\ww A \ww A \ww e} & A\ww A \ww A\ww A \ar[d]^{A\ww A \ww m_A}\\&A\ww A \ar[r]_-{\la\ww A}
&A\ww A& A\ww A \ww A \ar[l]^-{m_A \ww A}}\]
\begin{itemize}
  \item the top left triangle commutes since $\tau$ is symmetry;
  \item the middle rectangle commutes by naturality of $\tau$;
  \item the right rectangle commutes by naturality of composition;
  \item the curved triangle and the bottom right triangle commutes since $e$ is the unit of $\textbf{A}$, and
  \item the trapezoid commutes since $\la$ is an automorphism of the left $\textbf{A}$-module $(A,m)$,
\end{itemize}
while in the diagram
\[\xymatrix @C=0.6in @R=0.5in{A\ww A \ar@{=}[rdd]\ar[r]^-{A\ww A \ww e}& A\ww A
\ww A\ar[dd]^{A \ww m_A} \ar[r]^-{A\ww A \ww \la} & A\ww A \ww A \ar@{=}[d]\ar[rd]_{A\ww e \ww A \ww A }\ar[r]^-{A\ww A \ww e \ww A}&
A\ww A \ww A\ww A\ar[d]^{A\ww \tau_{A,\,A} \ww A}\\
& &A \ww A\ww A\ar[d]_{A \ww m_A}& A\ww A \ww A\ww A \ar[l]^{m_A\ww A \ww A }\\
&A\ww A \ar[r]_-{A \ww \la}&A\ww A& }\]
\begin{itemize}
  \item the top right triangle commutes since $\tau$ is symmetry;
  \item the left and the bottom right triangles commute since $e$ is the unit of $\textbf{A}$, and
  \item the rectangle commutes since $\la$ is an automorphism of the left $\textbf{A}$-module $(A,m)$,
  \end{itemize} it follows that
$${\ul{\textbf{Aut}}}\,_I^{\VV}(i_0)=\la\ww A\,\,\, \text{and}\,\,\,
{\ul{\textbf{Aut}}}\,_I^{\VV}(i_1)=A \ww \la$$ and since $\Delta_1={\ul{\textbf{Aut}}}\,_I^{\VV}(i_0)\cdot ({\ul{\textbf{Aut}}}\,_I^{\VV}(i_1))^{-1}=({\ul{\textbf{Aut}}}\,_I^{\VV}(i_1))^{-1} \cdot
{\ul{\textbf{Aut}}}\,_I^{\VV}(i_0)$, it follows that $\Delta_1=(A\ww \la^{-1})\cdot (\la \ww A)$.
Hence one has commutativity in
\[\xymatrix{\textbf{Aut}_{\V_\textbf{A}}(A)\ar@{=}[d]\ar[r]^-{\kappa\!_A} &\textbf{Aut}_{\textbf{A}-\texttt{cor}}(\mathcal{C}_e) \ar@{=}[d]&\\
\textbf{Aut}_{{_{\mathbf A}\!\!\V}}(A)\ar[r]_-{\Delta_1} &\mathcal{Z}^1(\iota\,, \ul{\textbf{Aut}}^{\VV}\!(I))&**[l]\,\,\,\,\,\,\,\,\,\subseteq\textbf{Aut}_{{_{\textbf{A}\ww \textbf{A}} \!\V}}(A\ww A),}\]
implying that $\textbf{Coker}(\kappa_A) \simeq \mathcal{H}^1(e\,, \ul{\textbf{Aut}}\,_I^{\VV}).$
This completes the proof of the theorem.

\end{proof}

It is well-known (e.g., \cite{Jon}) that for any commutative $\V$-algebra $\mathbf A$, one has
$$\mathcal{E}\!\downarrow \!\mathbf A=(\mathbf A \!\downarrow \!\verb"CAlg"(\V))^{\text{op}}.$$
Moreover, the co-slice category $\mathbf A \!\downarrow \!\verb"CAlg"(\V)$ is isomorphic to the category
$\verb"CAlg"({_{\mathbf A} \!\!\V})$. In other words, to give a commutative monoid $\mathbf B$ in the symmetric
monoidal category $ {_{\mathbf A}\!\V}$ is to give a morphism $\iota :\mathbf A\to \mathbf B$ of commutative monoids in $\V$.
The latter morphism serves as the unit morphism of the ${_{\mathbf A} \!\V}$-monoid $\mathbf B$. Write $(\iota \,)$
for the  corresponding commutative monoid in the symmetric monoidal category ${_{\mathbf A} \!\!\V}$. Then
a  (left) $(\iota \,)$-module  in  ${_{\mathbf A} \!\V}$  consists  of  a  (left) $\textbf{A}$-module  structure
$A\ww M \to M$  together  with  a morphism  $\rho:B \ww_\textbf{A}M \to M$ in ${_{\mathbf A} \!\V}$.
A straightforward  calculation  shows that the composite $$B\ww M \xr{q_{B,M}}B \ww_\textbf{A}M \xr{\rho} M$$
makes $M$  into  a $\textbf{B}$-module. In other direction, if $\varrho:B \ww M \to M$ is a $\textbf{B}$-module
structure on $M$, then the pair $(M, A\ww M \xr{\iota \ww M}  B \ww M \xr{\varrho} M)$ is a left $\textbf{A}$-
module and $\varrho=\varrho'\cdot q_{B,M}$ for a unique $\varrho':B \ww_\textbf{A}M \to M$. Then $(M, \varrho')$
is a left $(\iota \,)$-module in $_{\mathbf{B}}\!\V$. It  is  easily  checked  that  the above constructions  are
inverse  to  each other, and hence give an isomorphism ${_{(\iota )}\!({_{\mathbf A}\!\V})}\simeq {_{\mathbf{B}}\!\V}$
of categories. This allows us to identify the change-of-base  functor
$\iota^*=B \ww_\textbf{A}-: {_{\mathbf A}\!\V} \to {_{\mathbf B}\!\V}$ with the functor
$(\iota \,)\ww_\textbf{A}-: {_{\mathbf A}\!\V} \to {_{(\iota \,)}\!({_{\mathbf A}\!\V})}.$

One then constructs an $\mathcal{E}\!\downarrow \! \textbf{A}$-indexed category
$$\mathfrak {Alg}/\textbf{A}: (\mathcal{E}\!\downarrow \! \textbf{A})^{\text{op}} \to \mathbf{CAT} $$ as follows:
If $(\iota:\textbf{A}\to \textbf{B})$ is an object of $(\mathcal{E}\!\downarrow \! \textbf{A})^{\text{op}}$, then
$\mathfrak {Alg}/\textbf{A}(\iota)={_{\mathbf{B}}\!\!\V}$ and if \[\xymatrix @C=0.5in @R=0.02in{ &\textbf{B}\ar[dd]^{f}\\
\textbf{A}\ar[ru]^-{\iota} \ar[rd]_-{\iota'}& \\
&\textbf{B}' }\] is a morphism in $(\mathcal{E}\!\downarrow \! \textbf{A})^{\text{op}}$, then $f^*$ is the
functor $B'\ww \,_\textbf{B}-: {_{\mathbf{B}}\!\V} \to {_{\mathbf{B'}}\!\V}$. This $\mathcal{E}\!\downarrow \! \textbf{A}$-
indexed category satisfies the Beck-Chevalley condition (see the Appendix) and applying Theorem \ref{Amitsur.alg} gives:

\begin{theorem} \label{Amitsur.alg.gen.}Suppose that $\iota:\textbf{A}\to \textbf{B}$ is a morphism of commutative $\V$-algebras
such that the change-of-base functor $B \ww_\textbf{A} \,-:{_{\emph{\textbf{A}}}\!\V}\to {_{\emph{\textbf{B}}}\!\V}$
is comonadic. Then there is a natural isomorphism
$$\mathcal{H}^1(\iota\,, \ul{\emph{\textbf{Aut}}}_A^{\mathfrak {Alg}})\simeq \emph{\textbf{Ker}}(\emph{\textbf{Pic}}^{c}(\iota)).$$
\end{theorem}

 Let $R \subseteq A$ be an extension of commutative rings.  If $\iota :R \to A$ denotes the inclusion map, then $\mathcal{H}^1(\iota\,, \ul{\au}_R^{\mathfrak {Alg}})$ is just the first Amitsur cohomology group $ H^1(A/R,  U)$,
where  $U$ denotes  the  ``units"  functor.  
When $\iota $ is a  faithfully flat extension of commutative rings, then the change-of-base functor
$A \ww_R \,-:{_R\textbf{Mod}}\to {_A\textbf{Mod}}$  is comonadic (see, for example, \cite{Me}). Moreover,
 Specializing Theorem \ref{Amitsur.alg.gen.} to this case gives the following
well-known result (see, for example, \cite[Corollary  4.6]{CR}):

\begin{corollary} Let $A$ be a faithfully flat  commutative $R$-algebra and  let
$\iota :  R \to A$  be  the  inclusion  map.  Then  there  is  a  natural  isomorphism
$$ H^1(A/R,  U) \to  \emph{\textbf{Ker}}(\emph{\textbf{Pic}}^{c}(\iota)).$$
\end{corollary}

Theorem \ref{Amitsur.alg.gen.} also implies, in view of Remark \ref{sufficientcomonadicity}:

\begin{corollary}
Let $A$ be a separable commutative $R$-algebra and  let
$\iota :  R \to A$  be  the  inclusion  map.  Then  there  is  a  natural  isomorphism
$$ H^1(A/R,  U) \to  \emph{\textbf{Ker}}(\emph{\textbf{Pic}}^{c}(\iota)).$$
\end{corollary}

\subsection{Bicomodules} As observed in Subsection \ref{duality}, there are dual versions of the exact sequences built in Subsections \ref{autoinvert}, \ref{exactPic} and \ref{comonadicity} from an adjunction in a bicategory. One reason of recording explicitly them is to have statements tailored to concrete situations, as Examples \ref{firmfirst} and \ref{firmsecond} illustrate. Next, with the same motivation, we will consider the bicategory of bicomodules, and we will record some exact sequences derived from an adjunction in this setting. We close with some applications.

Suppose that $\V=(V, \ww, I)$ is a monoidal category with
equalizers such that all the functors $X \ww - : V \to V$ as well as $- \ww  X : V \to V$ for $X \in \V$, preserve equalizers. Coalgebras and (\emph{left, right, bi-}) \emph{comodules} in $\V$ can be defined as algebras
and left (right, bi-) modules in the \emph{opposite} monoidal category $(\V^{\text{op}}, \otimes, I)$.
The resulting categories are denoted by $\verb"Coalg"(\V)$, $^\textbf{C}\!\V$, $\V^\textbf{C}$ and
$^\textbf{C}\V^{\textbf{D}}$, $\textbf{C}$ and $\textbf{D}$ being coalgebras in $\V$.

Let $\textbf{C}, \textbf{D}, \textbf{E}$ be $\V$-coalgebras.  Dualizing the tensor product of bimodules, the \emph{cotensor product} $X \Box_{\textbf{C}} Y$
of a $(\textbf{D},\textbf{C})$-bicomodule $(X, \vartheta^l,\vartheta^r)$ and a $(\textbf{C},\textbf{E})$-bicomodule
$(Y, \theta^l, \theta^r)$ over $\textbf{C}$ is defined to be the equalizer of the pair of morphisms
$$\xymatrix {X \Box_{\textbf{C}} Y \ar[r]^-{\kappa_{X,Y}}&X\ww Y \ar@{->}@<0.5ex>[r]^-{\theta^r \ww 1} \ar@ {->}@<-0.5ex> [r]_-{1
\ww \vartheta^l}& X\ww C \ww Y .}$$

Note that $X \Box_{\textbf{C}} Y$ is a $(\textbf{D},\textbf{E})$-bicomodule.

\bigskip

Recall (for example, from \cite{DMS}) that there is a bicategory $\textbf{Bicom}(\V)$, called the bicategory of $\V$-bicomodules,
in which
\begin{itemize}
  \item 0-cells are $\V$-coalgebras,
  \item for $\textbf{C}, \textbf{D}\in \verb"Coalg"(\V)$, the hom-category
$\textbf{Bicom}(\V)(\textbf{C}, \textbf{D})$ is the category $^{\textbf{C}}\V^\textbf{D}$ of
$(\textbf{C},\textbf{D})$-bicomodules,
  \item 2-cells are morphisms of bicomodules, with obvious vertical composition and identities, and
  \item horizontal composition is the opposite of the cotensor
  product of bicomodules; the identity 1-cell $\iota_{\textbf{C}}$ for $\textbf{C} \in \verb"Coalg"(\V)$, is the
  \emph{regular} $(\textbf{C},\textbf{C})$-bicomodule $\textbf{C}$, i.e. $(\textbf{C},\textbf{C})$-bicomodule
  $(\textbf{C}, \delta_{\textbf{C}},\delta_{\textbf{C}})$.
\end{itemize}

\medskip
Suppose in addition that $\V$ admits, besides equalizers, finite colimits and coimage factorizations.
In this case, for any two $\V$-coalgebras $\textbf{C}$ and $\textbf{C}'$, the category
$^{\textbf{C}}\V^{\textbf{C}'} =\textbf{Bicom}(\V)(\textbf{C}, \textbf{C}')$, being the
category of Eilenberg-Moore algebras for the comonad $\textbf{C}\ww -\ww \textbf{C}'$, also admits
coequalizers (see, for example, \cite{FB}) and coimage factorizations (see, \cite{A}).

\medskip

Applying Theorem \ref{exact.d.} gives:

\begin{theorem} \label{exact.d.coalg.}Suppose that $\V$ admits finite colimits, coimage factorizations and  equalizers that are
preserved by the tensor product. Let $\textbf{C}$ and $\textbf{D}$ be $\V$-coalgebras and
$\Lambda: \textbf{D}\nr \textbf{C}$ be a 1-cell admitting a right adjoint $\Lambda \!^*: \textbf{C} \nr \textbf{D}$
with unit $\eta_{\Lambda}:\iota_{\textbf{D}}\to \Lambda \!^* \w \Lambda=\Lambda \Box_{\textbf{C}}\Lambda \!^*$ and
counit $\e_{\Lambda}: \Lambda \w \Lambda \!^*= \Lambda\!^* \Box_{\textbf{D}}\Lambda \to \iota_{\textbf{C}}.$
If $\e_{\Lambda}$ is epimorphic, then
$$1 \xr{}\emph{\textbf{Aut}}_{^{\textbf{C}}\V^{\textbf{C}}}
(C) \xr{\widehat{\omega}_0}\emph{\textbf{Aut}}_{^{\textbf{D}}\V^{\textbf{C}}}
(\Lambda) \xr{\ol{\mathcal{D}}_{\Lambda\!^*}} \mathcal{Q}^\textbf{C}_{\Lambda}\xr{\Omega_{\Lambda\!^*}}
\emph{\textbf{Pic}}(\textbf{C})$$ is an exact sequence of groups, while
$$\mathcal{Q}^\textbf{C}_{\Lambda}\xr{\Omega_{\Lambda\!^*}}\emph{\textbf{Pic}}(\textbf{C})
\xr{[- \Box_\textbf{C} {\Lambda\!^*}]} \pi_0(^{\textbf{C}}\V^\textbf{D})$$ is an exact sequence of pointed sets.
\end{theorem}

Each morphism $\phi : \textbf{D} \to \textbf{C}$ of $\V$-coalgebras determines a bicomodule $\textbf{D}_\phi: \textbf{D} \nr \textbf{C} $
defined to be $D$ together with the coactions $$(D \xr{\delta_D} D \ww D, D \xr{\delta_D} D \ww D \xr{D \ww \phi} D \ww C)$$
and a bicomodule $\textbf{D}^\phi: \textbf{C} \nr \textbf{D} $
defined to be $C$ together with the coactions $$(D \xr{\delta_D} D \ww D \xr{\phi \ww D} C \ww D, D \xr{\delta_D} D \ww D).$$
$\textbf{D}^\phi$ is right adjoint to $\textbf{D}_\phi$ in $\textbf{Bicom}(\V)$ with
$$\textbf{D}_\phi \w \textbf{D}^\phi = D \square_\textbf{D} D \simeq D \xr{\phi} C $$
as counit and $$D \xr{\ol{\delta}_D} D\square_\textbf{C} D=\textbf{D}^\phi \w \textbf{D}_\phi$$ as unit.
Here $D \xr{\ol{\delta}_D} D\square_\textbf{C} D$ is the unique morphism making the triangle
\[\xymatrix{D \ar[r]^-{\delta_D}\ar[rd]_{\ol{\delta}_D}& D \ww D \ar[d]^{\kappa_{D,D}}\\
&D\square_\textbf{C} D}\] commute.

It follows that a morphism $\phi: \textbf{D} \to \textbf{C}$ of $\V$-coalgebras gives rise to two functors
$$\phi^*=-\square_\textbf{C}D^\phi: {\V^{\mathbf{C}}} \to {\V^{\mathbf{D}}}$$
$$(Y, \vartheta_Y)\in {\V^{\mathbf{C}}} \longmapsto (Y\square_\textbf{C} D, Y \square_\textbf{C} \delta_D),$$
known as the \emph{change-of-cobase functor}, and
$$\phi_*=-\square_\textbf{D}D_\phi: {\V^{\mathbf{D}}} \to {\V^{\mathbf{C}}},$$ where
for any right $\textbf{D}$-comodule $(X, \theta_X)$,  $\phi_*(X, \theta_X)=X $ is a right $\textbf{C}$-comodule via the coaction
$$X \xr{\theta_X}  X\otimes D \xr{X \ww \phi} X \ww C;$$

\medskip

Since $\textbf{D}^\phi$ is right adjoint to $\textbf{D}_\phi$ in $\textbf{Bicom}(\V)$, it follows that $\phi^*$ is left adjoint to $\phi_*$.

\medskip

Assume further that  $\V$ is symmetric with symmetry $\tau$.
It is well-known that the category of cocommutative $\V$-coalgebras, $\verb"CCoalg"(\V)$, has pullbacks and
they are constructed as cotensor products: For any two morphism $\phi: \textbf{D} \to \textbf{C}$ and $\phi':
\textbf{D}' \to \textbf{C}$ in $\verb"CCoalg"(\V)$, the diagram
$$\xymatrix{{\mathbf{D}'\square_\textbf{C}\textbf{D}} \ar[r]^{p_{_\textbf{D}}} \ar[d]_{p_{\textbf{D}'}} & {\mathbf{D}} \ar[d]^{\phi}\\
{\mathbf{D}'} \ar[r]_{\phi'}& {\mathbf{C}}\,,}$$ where $p_{\textbf{D}}=\e_{D'}\square_\textbf{C}\textbf{D}$ and
$p_{\textbf{D}'}=\textbf{D}'\square_\textbf{C}\e_D$, is a pullback in $\verb"CCoalg"(\V)$.

Define a $\verb"CCoalg"(\V)$-indexed category $$\W: (\verb"CCoalg"(\V))^{\text{op}} \to \mathbf{CAT} $$ by setting
$$\W(\textbf{C})={\V^{\mathbf{C}}} \quad \text{and}\quad  \W(\textbf{D} \xr{\phi} \textbf{C})=
{\V^{\mathbf{C}}} \xr{\phi^*}{\V^{\mathbf{D}}},$$ where
$\phi^*=-\square_\textbf{C} D:\V^{\mathbf{C}} \to {\V^{\mathbf{D}}}$ is the change-of-cobase functor. Since for any morphism
${\phi: \textbf{D} \to\textbf{C}}$ in $\verb"CCoalg"(\V)$,
the functor $\phi^*=-\square_\textbf{C} D$ admits as a left adjoint the
forgetful functor $\V^{\mathbf{D}} \xr{\phi_*=-\square_\textbf{D} D}{\V^{\mathbf{C}}}$,
the $\verb"CCoalg"(\V)$-indexed category $\W$ admits coproducts iff it satisfies the Beck-Chevalley condition,
i.e., for any morphism $\phi':\textbf{D}' \to \textbf{C}$ in  $\verb"CCoalg"(\V)$, the diagram
$$\xymatrix{\V^{\mathbf{D}'} \ar[r]^{(p_{\textbf{D}'})^*} \ar[d]_{(\phi')_*} & \V^{\mathbf{D}'\square_\textbf{C}\textbf{D}} \ar[d]^{(p_{\textbf{D}})_*}\\
\V^{\mathbf{C}} \ar[r]_{\phi^*}& \V^{\mathbf{D}}}$$ commutes up to canonical isomorphism.
It is easily to check that this condition is equivalent to saying that for any $X \in \V^{\mathbf{D}'}$,
one has an isomorphism $$X \square_{\textbf{D}'}({D'\square_\textbf{C}D}) \simeq X\square_\textbf{C}D,$$
and this is certainly the case, since the tensor product in $\V$ preserves reflexive equalizers
by our assumption on $\V$. Thus, $\W$ admits coproducts.

As in the case of algebras, given a cocommutative $\V$-coalgebra $\textbf{C}$, we get the functor
$$\ul{\textbf{Aut}}^{\W}_{\,\textbf{C}}:(\verb"CCoalg"\!\downarrow \!\mathbf C)^{\text{op}} \to \textbf{Group}$$
sending  each object $\phi: \textbf{D}\to \textbf{C}$ of $(\verb"CCoalg"\!\downarrow \!\mathbf C)^{\text{op}}$ (i.e. a morphism
$\phi: \textbf{D}\to \textbf{C}$ of cocommutative $\V$-coalgebras) to the group
\begin{equation}\label{ca.1}
\ul{\textbf{Aut}}_{\,\textbf{C}}^{\W}(\textbf{D}\xr{\phi} \textbf{C})=\textbf{Aut}_{ \V^{\mathbf D}}(D).
\end{equation}
On morphisms it is defined in the following way: Given a morphism $f:(\textbf{D},\phi) \to (\textbf{D}',\phi')$ in
$(\verb"CCoalg"\!\downarrow \!\mathbf C)^{\text{op}}$ (i.e., a commutative diagram
\[\xymatrix{& \textbf{C} &\\ \textbf{D}' \ar[ru]^{\phi'} \ar[rr]_f&& \textbf{D} \ar[lu]_{\phi} }\] in $\verb"CCoalg"\!\downarrow \!\mathbf C$), then
\begin{equation}\label{ca.2}
\ul{\textbf{Aut}}_{\,\textbf{C}}^{\W}(f:(\textbf{D},\phi) \to (\textbf{D}',\phi'))(\sigma)=(\e_D \ww D')\cdot(\sigma \ww D') \cdot(\phi \ww D')\cdot \delta_{{D}'}
\end{equation} for all $\sigma \in \textbf{Aut}_{ \V^{\mathbf D}}(D)$.

Note that, since for each morphism $\phi:\textbf{D} \to \textbf{C}$ of cocommutative $\V$-coalgebras, $\ul{\textbf{Aut}}^{\W}_{\,\textbf{C}}(\phi)$
is an abelian group, it follows that the functor $\ul{\textbf{Aut}}_{\,\textbf{C}}^{\W}$ also takes values in the category
of abelian groups.

Given a morphism $\phi : \textbf{D} \to \textbf{C}$ of cocommutative $\V$-coalgebras, consider the associated augmented simplicial object
\[\xymatrix{(\textbf{D} /\textbf{C} )_*: \ar@{}[r]|(.65){.\,.\,.}  & D^{2}  \ar[rr]|-{\pa_1} \ar@<2mm>[rr]^-{\pa_0}
  \ar@<-2mm>[rr]_-{\pa_2}&& D^{1} \ar@<-1mm>[r]_-{\pa_1} \ar@/_2pc/ [ll]_{s_0, s_1} \ar@<1mm>[r]^-{\pa_0} & D^{0} \ar@/_2pc/ [l]_{s_0}\ar[r]^-{\phi} & C
}\]  where
\begin{itemize}
  \item $D^{0}=D$
  \item $D^{n}=\underbrace{D\square_\textbf{C}D \cdots \square_\textbf{C}D}_{(n+1)-\, \text{times}}$ \,\,\,for all $n\geq 1$
  \item $\pa_i= D^{i}\square_\textbf{C} \phi\square_\textbf{C}D^{n-i}:D^{n+1} \to D^{n}$ \,\,\,for all $0\leq i \leq n$
  \item $s_j=D^{j}\square_\textbf{C} \ol{\delta}_D\square_\textbf{C}D^{(n-j-1)}:D^{(n-1)}\to D^{(n)}$ \,\,\,for all $0\leq j \leq n$.
\end{itemize}

Applying the functor $\ul{\textbf{Aut}}^{\W}_{\,\textbf{C}}:(\verb"CCoalg"\!\downarrow \!\mathbf C)^{\text{op}} \to \textbf{Group}$ to
$(\textbf{D} /\textbf{C} )_*$, we obtain the following augmented cosimplicial group
\[\xymatrix{(\textbf{D} /\textbf{C}, \ul{\textbf{Aut}}_{\,\textbf{C}}^{\W})_*: \textbf{Aut}_{ \V^{\mathbf C}}(C) \ar[rr]^-{\ul{\textbf{Aut}}_{\,\textbf{C}}^{\W}(\phi)} &&
\textbf{Aut}_{ \V^{\mathbf D}}(D) \ar@<-2mm>[rr]_-{\ul{\textbf{Aut}}_{\,\textbf{C}}^{\W}(\pa_1)}  \ar@<2mm>[rr]^-{\ul{\textbf{Aut}}_{\,\textbf{C}}^{\W}(\pa_0)}
&& \textbf{Aut}_{ \V^{\mathbf D \square_\textbf{C}\!\mathbf D }}(D \square_\textbf{C} D) \ar@/_3pc/ [ll]_{\ul{\textbf{Aut}}_{\,\textbf{C}}^{\W}(s_0)} {.\,.\,.}&}\] Since for any $n \geq 0$, $\textbf{D}^{n}$ is a cocommutative $\V$-coalgebra, all the categories $\V^{\mathbf D^n}$ are symmetric monoidal
with monoidal unit $D^n$, it follows  that $(\textbf{D} /\textbf{C}, \ul{\textbf{Aut}}_{\,\textbf{C}}^{\W})_*$
is in fact an augmented abelian cosimplicial group.

Write $\textbf{Pic}^{c}(\textbf{C})$ for the subgroup of $\textbf{Pic}(\textbf{C})$ consisting of all classes of invertible
$(\textbf{C},\textbf{C})$-bicomodules $(X, \vartheta_l:X \to C \ww X, \vartheta_r: \to X \ww C)$ such that
$\rho_r=\tau_{C,X}\cdot \rho_l$. Then $\textbf{Pic}^{c}(\textbf{C})$ is an abelian group. Moreover,
given a morphism $\phi: \textbf{D}\to \textbf{C}$ of cocommutative $\V$-coalgebras, the map
$$\textbf{Pic}^{c}(\phi):\textbf{Pic}^{c}(\textbf{C})\to \textbf{Pic}^{c}(\textbf{D})$$
defined  by $\textbf{Pic}^{c}(\phi)([P])=[P\Box_\textbf{C} D]$, is a homomorphism of abelian groups.

Now with the complex of abelian groups
$$0 \xr{} \textbf{Aut}_{ \V^{\mathbf C}}(C) \xr{\ul{\textbf{Aut}}_{\,\textbf{C}}^{\W}(\phi)}
\textbf{Aut}_{ \V^{\mathbf D}}(D) \xr{\Delta_1} \textbf{Aut}_{ \V^{\mathbf D \square_\textbf{C}\!\mathbf D }}
(D \square_\textbf{C} D) \xr{\Delta_2} \cdots$$
$$\Delta_n=\prod_{i=0}^{n} \textbf{Aut}_{\,\textbf{C}}^{\W}(\pa_n)^{(-1)^n}, \, n\geq 1,$$
corresponding to the augmented abelian cosimplicial group $(\textbf{D} /\textbf{C}, \ul{\textbf{Aut}}_{\,\textbf{C}}^{\W})_*\,,$
by arguments similar to those used in the of Theorems \ref{Amitsur.alg.gen.}, we derive
the following result.

\begin{theorem} \label{Amitsur.coalg.gen.} Suppose that $\phi:\mathbf{D}\to \mathbf{C}$ is a morphism of cocommutative
$\V$-coalgebras such that the change-of-cobse functor $\phi^*=-\square_\textbf{C}D: {\V^{\mathbf{C}}} \to {\V^{\mathbf{D}}}$
is monadic. Then there is a natural isomorphism
$$\mathcal{H}^1(\phi\,, \ul{\mathbf{Aut}}_{\,\mathbf{C}}^{\W})\simeq \mathbf{Ker}(\mathbf{Pic}^{c}(\phi)).$$
\end{theorem}

Specializing Theorem \ref{Amitsur.coalg.gen.} to the case where $\V=\verb"Vect"_k$ is the category of vector spaces over a field $k$
and using that for any cocommutative $k$-coalgebra $\textbf{C}$, $\textbf{Pic}^{c}(\textbf{C})=0$ (see, for example,
\cite[Proposition 3.2.14]{CVO} or \cite[Proposition 4.1]{CGTV}), we get the following version of Hilbert's theorem 90 for cocommutatvie 
coalgebras:

\begin{theorem} \label{Hilbert} Suppose that $\phi:\mathbf{D}\to \mathbf{C}$ is a morphism of cocommutative $k$-coalgebras
such that the change-of-cobase functor $-\square_\mathbf{C}D: {\verb"Vect"_k^{\mathbf{C}}} \to {\verb"Vect"_k^{\mathbf{D}}}$
is monadic. Then
$$\mathcal{H}^1(\phi\,, \ul{\mathbf{Aut}}_{\,\mathbf{C}}^{\W})=0.$$
\end{theorem}

\appendix

\section{Some classical results and constructions}\label{Appendix}

Let $\A$ be a category with pullbacks. An $\A$-indexed category $\X$ is
a pseudo-functor $\A^{\text{op}} \to \mathbf{CAT}$,
where $\mathbf{CAT}$ denotes the 2-category of locally small (but possibly large) categories,
explicitly given by the data of a family of categories $\X(a)$,
indexed by the objects of $\A$, with \emph{change of base functors} $$\iota^*:\X(a) \to \X(b)$$ for each
morphism $\iota:b \to a$ of $\A$ and with additional structure expressing the idea of a
pseudo-functor (see \cite{McP}, \cite{PS}).

A simple but important example of an $\A$-indexed category is the so-called \emph{basic}
$\A$-\emph{indexed category} $\A\!\downarrow \!-:\A^{\text{op}} \to \mathbf{CAT}$ that to any object $a \in \A$
associates the slice category $\A \!\downarrow \!a$, and to a morphism $\iota:b \to a$  the functor
$\iota^* : \A\!\downarrow \! a \to \A \!\downarrow \! b$ given by pulling back along $\iota$.

Fix and object $a$ of  $\A$, and $\X : \A^{\text{op}} \to \mathbf{CAT}$ an $\A$--indexed category. For each object $x \in \X(a)$,  let us define a functor
$$\ul{\textbf{Aut}}^\mathcal{X}_{\,x}:(\A \!\downarrow \! a)^{\text{op}} \to \textbf{Group}$$
sending  each object $b \xr{\kappa} a$ of $\A \!\downarrow \! a$ to the group
\begin{equation}\label{AutX}
\ul{\textbf{Aut}}^\mathcal{X}_{\,x}(\kappa)\stackrel{\text{def}}=\textbf{Aut}_{\X(b)}(\kappa^*(x)).
\end{equation}

\subsection{Descent} Consider the augmented simplicial complex
\begin{equation}\label{Amitsurg}\xymatrix{(b/a)_*: \ar@{}[r]|(.65){.\,.\,.}  & (b/a)_2  \ar[rr]|-{\pa_1} \ar@<2mm>[rr]^-{\pa_0}
  \ar@<-2mm>[rr]_-{\pa_2}&& (b/a)_1 \ar@<-1mm>[r]_-{\pa_1} \ar@/_2pc/ [ll]_{s_0, s_1} \ar@<1mm>[r]^-{\pa_0} & (b/a)_0 \ar@/_2pc/ [l]_{s_0}\ar[r]^-{\iota} & a
}\end{equation}
 associated to any morphism $\iota : b \to a$ in $\A$,  where
\begin{itemize}
  \item $(b/a)_0=b$
  \item $(b/a)_n=\underbrace{b \times_a b \times_a \cdots \times_a b}_{(n+1)- \text{times}}$ \,\,\,for all $n\geq 1$
  \item $\pa_i=<p_1, p_2, ..., p_{i-1}, p_{i+1}, ..., p_{n+1}>:a_{n} \to a_{n-1}$ \,\,\,for all $0\leq i \leq n$
  \item $s_j=\underbrace{b \times_a b \times_a \cdots \times_a b}_{j- \text{times}}\times_a \,\Delta_{b/a} \times_a\underbrace{b \times_a b \times_a \cdots \times_a b}_{(n-j-1)- \text{times}}:(b/a)_{n-1} \to (b/a)_n$ \,\,\,for all $0\leq j \leq n$.
\end{itemize} Here $p_{i}:\underbrace{b \times_a b \times_a \cdots \times_a b}_{(n+1)- \text{times}}\to b$ is the projection to the $i$-th factor, while $\Delta_{b/a}$ is the diagonal morphism $b \to b \times_a b$.

Let us recall from \cite{Gr}  the definition of  the category $\mathrm{Des}_\X(\iota)$ of
$\X$-\emph{descent data relative to} $\iota$. Its objects are pairs $(x, \vartheta)$,
with $x$ an object of $\X(b)$ and $\vartheta: \pa_1^*(x) \simeq \pa_0^*(x)$ an isomorphism in $\X(b\times_a b)$
such that $s_0^*(\vartheta)=1$ and the diagram
$$
\xymatrix{\pa_{2}^*\pa_1^*(x) \ar[r]^{\pa_{2}^*(\vartheta)} \ar[d]_{\simeq}& \pa_{2}^*\pa_0^*(x)
\ar[r]^{\simeq}& \pa_{0}^*\pa_1^*(x) \ar[d]^{\pa_{0}^*(\vartheta)}\\
\pa_{1}^* \pa_1^*(x) \ar[r]_{\pa_{1}^*(\vartheta)}& \pa_{1}^* \pa_0^*(x) \ar[r]_{\simeq} & \pa_{0}^*\pa_0^*(x)}
$$ commutes in $\X(b\times_a b\times_a b)$. Here the labeled isomorphisms are the canonical ones of the
$\A$-indexed category $\X$ coming from the simplicial identities
$$\pa_i \pa_j = \pa_{j-1} \pa_i \,\,\,\, (i < j).$$

A morphism $f:(x, \vartheta) \to (y, \theta)$ in $\mathrm{Des}_\X(\iota)$ is a morphism
$f:x \to y$ in $\X(b)$ which commutes with the descent
data $\vartheta$ and $\theta$ in the sense that the diagram
$$
\xymatrix{\pa_1^*(x) \ar[r]^{\vartheta} \ar[d]_{\pa_1^*(f)} & \pa_0^*(x) \ar[d]^{\pa_0^*(f)}\\
\pa_1^*(y) \ar[r]_{\theta}& \pa_0^*(y)}
$$ commutes in $\X(b\times_a b)$.

If $z$ is an object of $\X(a)$, then $\iota^*(z)$ comes equipped with \emph{canonical descent datum} given by the composite
$$\pa_1^*(\iota^*(z))\simeq (\iota \pa_1)^*(z)=(\iota \pa_0)^*(z)\simeq \pa_0^*(\iota^*(z))$$ of canonical isomorphisms. In other words,
the functor $\iota^*$ factors as
$$
\xymatrix{\X(a) \ar[rd]_{\iota^*} \ar[r]^-{K_\iota} & \mathrm{Des}_\X(\iota) \ar[d]^{U}\\
& \X(b)}
$$ where $U$ is the evident forgetful functor, and $K_\iota$ sends $z \in \X(a)$ to $\iota^*(z)$ equipped with the canonical descent datum.

\begin{definition}
$\iota$ is called an $\X$-\emph{descent morphism} if $K_\iota$ is full and faithful, and an \emph{effective} $\X$-\emph{descent morphism} if
$K_\iota$ is an equivalence of categories.
\end{definition}

Let $x \in \X(b)$. Write $\mathrm{Des}_\X(x)$ for the set of all descent data on $x$. Two descent data $(x, \vartheta)$ and
$(x, \vartheta')$ on $x$ are called \emph{equivalent} if they are isomorphic objects in the category $\mathrm{Des}_\X(\iota)$.
The set of equivalence classes of descent data on $x$ is denoted by $\textbf{Des}_\X(x)$. If $x=\iota^*(y)$ for some
$y \in \X(a)$, then $\textbf{Des}_\X(\iota^*(y))$ is a pointed set with the class of canonical descent datum as a distinguished element.

\begin{definition}\label{havingproducts}
An $\A$-indexed category $\X$ \emph{has products} (resp. \emph{coproducts}) if for each morphism $\iota:b \to a$ in $\A$,
the change of base functor $\iota^*:\X(a) \to \X(b)$ admits a right (resp. left) adjoint $\Pi_\iota: \X(b) \to \X(a)$
(resp. $\Sigma_\iota: \X(b) \to \X(a)$) and the \emph{Beck-Chevalley condition} is satisfied, i.e.,  for every pullback diagram
$$
\xymatrix{c \ar[r]^{q} \ar[d]_{p} & b \ar[d]^{\iota}\\
b' \ar[r]_{\iota '}& a\,,}$$ the  following  diagram
$$
\xymatrix{\X(b) \ar[r]^{q^*} \ar[d]_{\Pi_\iota} & \X(c) \ar[d]^{\Pi_p}\\
\X(a) \ar[r]_{(\iota ')^*}& \X(b')}\qquad \,\,\, \text{(resp.}\,\,
\xymatrix{\X(b') \ar[r]^{p^*} \ar[d]_{\Sigma_{\iota'}} & \X(c) \ar[d]^{\Sigma_q}\\
\X(a) \ar[r]_{\iota^*}& \X(b)})$$ commutes up to canonical isomorphism.
\end{definition}

We shall need the following version of the B\'{e}nabou-Roubaud-Beck theorem (cf. \cite[Proposition B1.5.5]{Jon}):

\begin{theorem}\label{BRB} For an $\A$-indexed category $\X:\A^{\text{op}} \to \mathbf{CAT}$ having
products (resp. coproducts) and for an arbitrary morphism $\iota:b \to a$ in $\A$,
the category  $\mathrm{Des}_\A(\iota)$ of descent data with respect to $\iota$ is
isomorphic to the Eilenberg-Moore category of coalgebras (resp. algbras) for the
comonad  (resp. monad) $\textbf{G}_\iota$ (resp. $\textbf{T}_\iota$) on $\X(b)$ generated by the adjoint pair
$\iota^* \dashv \Pi_\iota:\X(b) \to \X(a)$  (resp. $\Sigma_\iota \dashv \iota^*:\X(a) \to \X(b)$. Moreover,
modulo this equivalence, the functor $K_\iota : \X(a) \to \mathrm{Des}_\X(\iota)$  corresponds to the comparison
functor $\X(a) \to (\X(b))^{\textbf{G}_\iota}$  (resp. $\X(a) \to (\X(b))_{\textbf{T}_\iota}$). Thus, $\iota$ is
an effective $\X$-descent morphism if and only if the functor $\iota^*:\X(a) \to \X(b)$ is comonadic (resp. monadic).
\end{theorem}

\subsection{Amitsur cohomology} Let $F$ be a functor  on the category $(\A \!\downarrow \! a)^{\text{op}}$ with values in the category of groups.
 Applying $F$ to the augmented simplicial object (\ref{Amitsurg}), one gets a coaugmented  cosimplicial group
\[\xymatrix{(b / a,F)_*:\,\,F(a) \ar[r]^-{F(\iota)} & F(b) \ar@<-1mm>[r]_-{F(\pa_1)}  \ar@<1mm>[r]^-{F(\pa_0)}
& F(b\times_a b)\ar@/_2pc/ [l]_{F(s_0)}  \ar[rr]|-{F(\pa_1)}\ar@<2mm>[rr]^-{F(\pa_0)} \ar@<-2mm>[rr]_-{F(\pa_2)}&&
F(b\times_a b \times_a b) \ar@/_3pc/ [ll]_{F(s_0), \, F(s_1)}\ar@{}[r]|(.65){.\,.\,.}&}\] with cofaces $F(\pa_i)$
and codegeneracies $F(s_i)$, and hence one has the non-abelian 0-cohomology group $\mathcal{H}^0((b/a,F)_*$ and the non-abelian 1-cohomology pointed set
$\mathcal{H}^1((b/a,F)_*)$.  More precisely,
$\mathcal{H}^0((b/a,F)_*)$  is the equalizer of the pair $(F(\pa_0),F(\pa_1))$.

On the other hand, a 1-\emph{cocycle }is an element $x \in F(b\times_a b)$ such that
$$F(\pa_1)(x)=F(\pa_2)(x)\cdot F(\pa_0)(x)$$ in $F(b \times_a b \times_a b)$. Write $\mathcal{Z}^1((b/a,F)_*)$ for the set of 1-cocycles.
This set is pointed with point the unit element of $F(b \times_a b \times_a b)$. Two 1-cocycles $x$ and $x'$ are equivalent
if $$x'=F(\pa_1)(y)\cdot x \cdot F(\pa_0)(y)^{-1}$$ for some element $y \in F(b \times_a b)$.
This is an equivalence relation on $\mathcal{Z}^1(b/a,F)_*)$ and $\mathcal{H}^1((b/a,F)_*)$ is defined as the factor-set of equivalence classes of 1-cocycles, and it is a pointed set.

We call $\mathcal{H}^0((b/a,F)_*$ (resp. $\mathcal{H}^1((b/a,F)_*$) the
\emph{zeroth Amitsur cohomology group of} $\iota : a \to b$ \emph{with values in }$F$ (resp. the
\emph{first Amitsur cohomology pointed set of} $\iota : a \to b$ \emph{with values in }$F$) and denote it by
$\bH^0(\iota, F)$ (resp. $\bH^1(\iota, F)$).

If for each $n\geq 1$, $F(\underbrace{b \times_a b \times_a \cdots \times_a b}_{n \text{times}})$ is an abelian
group (which is indeed the case if the functor $F$ factors through the category of abelian groups),  it is possible
to define higher cohomology groups as follows. Let
$$C(\iota, F):0 \xr{} F(a) \xr{F(\iota)} F(b) \xr{\Delta_1} F(b\times_a b) \xr{\Delta_2} \cdots$$
$$\Delta_n=\prod_{i=0}^{n} (F(\pa_n))^{(-1)^n}, \, n\geq 1$$ be the complex of abelian groups associated
to the abelian cosimplicial group $(F/\iota)^*$. The cohomology groups of this complex are called the \emph{Amitsur cohomology
groups of} $\iota : a \to b$ \emph{with values in }$F$ and are denoted by $\bH^i(\iota \,, F)$.

The following result of Grothendieck is to be found in \cite{Gr}.
\begin{proposition}\label{Grothendieck} Let $\X: \A^{\text{op}} \to \mathbf{CAT}$ be an $\A$-indexed category,
$\iota:b \to a$ a morphism in $\A$ and $x \in \X(a)$. Then the assignment that takes
$\alpha \in \emph{\textbf{Aut}}_{\X(b\times_a b)}(\pa^*(x))$, where $\pa$ denotes the common value of
$\iota \cdot \pa_0$ and $\iota \cdot \pa_1$, to the composite
$$\pa^*_1(\iota^*(x))\simeq (\iota \cdot \pa_1)^*(x)=\pa^*(x) \xr{\alpha}\pa^*(x)\simeq (\iota \cdot \pa_0)^*(x)\simeq \pa^*_0(\iota^*(x))$$
yields an isomorphism
$$\Upsilon^{\iota,x}:\mathcal{Z}^1(\iota \,, \emph{\ul{\textbf{Aut}}}^\mathcal{X}_{\,x}) \simeq \mathrm{Des}_\X (\iota^*(x))$$ of the pointed sets.
When $\iota$ is an effective $\X$-descent morphism, $\Upsilon^{\iota,x}$ induces an isomorphism
$$\widehat{\Upsilon}^{\iota,x}:\mathcal{H}^1(\iota \,, \emph{\ul{\textbf{Aut}}}^\mathcal{X}_{\,x})\simeq \emph{\textbf{Des}}_\X(\iota)$$
of pointed sets.
\end{proposition}

\bibliographystyle{amsplain}

\end{document}